\documentclass[a4paper,11pt,twoside,reqno]{article}
\usepackage{amsmath, amsfonts, amssymb, amsthm, mathtools}
\usepackage{stmaryrd}
\usepackage{bbm}
\usepackage{cases}
\usepackage{hyperref}
\usepackage[T1]{fontenc}
\usepackage[latin1]{inputenc}
\usepackage{upgreek}
\usepackage[english]{babel}
\usepackage[cyr]{aeguill}
\usepackage{xspace}
\usepackage{graphicx}
\usepackage{caption}
\usepackage{color}
\usepackage[labelformat=empty]{caption}
\usepackage{epigraph}
\usepackage{etoolbox}
\usepackage{enumitem}
\usepackage{stmaryrd}
\usepackage[margin=1in]{geometry}
\usepackage{cite}
\usepackage{dsfont}
\usepackage{ushort}

\newcommand{\ol}{\overline}
\newcommand{\ul}{\underline}
\newcommand{\eps}{\varepsilon}

\newcommand{\ba}{\begin{array}}
\newcommand{\ea}{\end{array}}
\newcommand{\be}{\begin{equation}}
\newcommand{\ee}{\end{equation}}
\newcommand{\bea}{\begin{eqnarray}}
\newcommand{\eea}{\end{eqnarray}}
\newcommand{\beaa}{\begin{eqnarray*}}
\newcommand{\eeaa}{\end{eqnarray*}}

\def\dbE{\mathbb{E}}
\def\dbF{\mathbb{F}}

\def\dbH{\mathbb{H}}

\def\dbP{\mathbb{P}}
\def\dbR{\mathbb{R}}

\def\a{\alpha}
\def\b{\beta}

\def\l{\lambda}

\def\o{\omega}

\def\O{\Omega}
\def\cC{{\cal C}}

\def\cE{{\cal E}}
\def\cF{{\cal F}}
\def\cG{{\cal G}}

\def\cK{{\cal K}}
\def\cL{{\cal L}}
\def\cM{{\cal M}}
\def\cN{{\cal N}}

\def\cP{{\cal P}}

\def\cR{{\cal R}}

\def\cU{{\cal U}}

\def\cZ{{\cal Z}}

\def\no{\noindent}

\def\ms{\medskip}

\def\q{\quad}

\def\cd{\cdot}

\DeclareMathOperator{\Tr}{Tr}

\def\qed{ \hfill \vrule width.25cm height.25cm depth0cm\smallskip}

\newcommand{\basa}{\begin{assumption}}
\newcommand{\easa}{\end{assumption}}

\newcommand{\bas}{\begin{assum}}
\newcommand{\eas}{\end{assum}}

\def\limsup{\mathop{\overline{\rm lim}}}

 \def\cd{\cdot}

\def\1{{\bf 1}}

\def\:{\!:\!}

at 9pt

\newtheorem{thm}{Theorem}[section]
\newtheorem{lem}[thm]{Lemma}

\newtheorem{prop}[thm]{Proposition}
\newtheorem{rem}[thm]{Remark}

\newtheorem{assum}[thm]{Assumption}

\DeclareMathOperator*{\argmin}{\arg\!\min}
\DeclareMathOperator*{\argmax}{\arg\!\max}

\begin{document}

\renewcommand {\theequation}{\arabic{section}.\arabic{equation}}
\def\thesection{\arabic{section}}

\numberwithin{equation}{section}
\numberwithin{thm}{section}

	\title{\textbf{\huge Continuous-Time Principal-Agent Problem in Degenerate Systems}}
	\author{Kaitong HU \footnote{CMAP, Ecole Polytechnique, IP Paris, 91128 Palaiseau Cedex, France, kaitong.hu@polytechnique.fr.} \and  
		Zhenjie REN\footnote{Universit\'e Paris-Dauphine, PSL Research University, CNRS, UMR [7534], Ceremade, 75016 Paris, France, ren@ceremade.dauphine.fr.} \and
		Nizar TOUZI \footnote{CMAP, Ecole Polytechnique, IP Paris, 91128 Palaiseau Cedex, France, nizar.touzi@polytechnique.edu.}}
	\date{}
	\maketitle{}
	
\begin{abstract}
	In this paper we present a variational calculus approach to Principal-Agent problem with a lump-sum payment on finite horizon in degenerate stochastic systems, such as filtered partially observed linear systems. Our work extends the existing methodologies in the Principal-Agent literature using dynamic programming and BSDE representation of the contracts in the non-degenerate controlled stochastic systems. We first solve the Principal's problem in an enlarged set of contracts defined by a forward-backward SDE system given by the first order condition of the Agent's problem using variational calculus. Then we use the sufficient condition of the Agent's problem to verify that the optimal contract that we obtain by solving the Principal's problem is indeed implementable (i.e. belonging to the admissible contract set). Importantly we consider the control problem in a weak formulation. Finally, we give explicit solution of the Principal-Agent problem in partially observed linear systems and extend our results to some mean field interacting Agents case. 
\end{abstract}	

\paragraph*{Keywords:}Stochastic Control of non-Markovian systems, Stochastic Maximum Principles, path-dependent Forward-backward SDEs, Principal-Agent Problem, Contract Theory

\paragraph*{MSC:} 60H30, 91A23, 91A35 

\section{Introduction}
Moral hazard is one of the prime risks of systemic instability and inefficiency, as already pointed out by Adam Smith. Finding the optimal contract between two parties - The Principal and the Agent, when the Agent's effort cannot be observed therefore cannot be contracted upon, is a classical moral hazard problem in microeconomics. Applications can be widely found in corporate finance, portfolio management \cite{BD05} and more recently energy transition \cite{AEEH17, APT18}.

While the research on discrete-time models dated further back, the first continuous-time model was proposed in the seminal work of Holmstr{\"o}m and Milgrom \cite{HM87}, in which they study a simplified model with lump-sum payment on finite horizon while assuming the Agent controls only the drift of the state process. They show that the optimal contract should be linear in aggregate output when the Agent has CARA utility functions with a monetary cost of effort. Their work has been extended by Sch\"attler and Sung \cite{SJ93, SJ97}, Sung \cite{Sung95, Sung97}, M\"uller\cite{Muller98, Muller00}, Hellwig and Schmidt \cite{HS02}. Later Williams \cite{Williams09}, Cvitani\'c, Wan and Zhang \cite{CWZ2009} use the stochastic maximum principle and forward-backward stochastic differential equations (abbreviated FBSDE) to characterize the optimal compensation for more general utility functions. 

Many different continuous-time models have been proposed in the past few decades, notably the one introduced by Sannikov in his seminal paper \cite{Sannikov08}. He considered Principal-Agent problem on infinite horizon with continuous payment while allowing Principal to fire or retire the Agent at any time. From a mathematical perspective, his approach sheds light on the contracting relationship in Princpal-Agent's problem using dynamic programming and leads to simple computational procedure to find the optimal contract by, in his case, solving an ordinary differential equation. 

More recently, Sannikov's idea was reinterpreted and extended by Cvitani\'c, Possama\"i and Touzi \cite{Touzi} to a more general set-up with a more direct and easier approach. We shall illustrate their contribution in the following toy model. Denote by $\xi$ the contract paid at the end of a finite time interval $[0,T]$. Let's consider the following optimization for the Agent:
\beaa
\max_\a \dbE\left[\xi(X^\a) - \int_0^T c(\a_t) dt \right],\q\mbox{where}\q dX^\a_t = dW_t -\a_t dt .
\eeaa
The crucial observation in \cite{Touzi} is that both the contract $\xi$ and the Agent's best response $\a^*[\xi]$ can be characterized by the following backward stochastic differential equation (in short, BSDE, for readers not familiar with BSDE we refer to \cite{PP90, EKPQ97}, and in particular to \cite{CZ12} for the applications on the contract theory):
\beaa
dY_t = -c^*(Z_t) dt +Z_t dW_t,\q Y_T =\xi, \q\mbox{where}\q c^*(z) = \max_a \big\{az - c(a) \big\},
\eeaa
namely, $\xi = Y_T $ and $\a^*_t[Z] = \argmax_a \big\{aZ_t - c(a) \big\}$ for all $t\in [0,T]$. This induces a natural (forward) representation of the couple $(\xi, \a^*[\xi])$:
\beaa
\begin{cases}
	\xi = Y_T^{Y_0, Z} := \displaystyle Y_ 0 -\int_0^T c^*(Z_t) dt + \int_0^T Z_t dW_t  \\ 
	\a^*_t[\xi] := \a^*_t[Z] = \displaystyle\argmax_a \big\{aZ_t - c(a)\big\}, \,\, \mbox{ for all $t\in [0,T]$}
\end{cases}
\mbox{ for some } (Y_0,Z),
\eeaa
and this neat representation transforms the once puzzling principal's problem to be a classical stochastic control problem, namely,
\beaa
\max_\xi \dbE\Big[U\left(X^{\a^*[\xi]}_T - \xi\right)\Big] = \max_{Y_0,Z} \dbE\Big[U\left(X^{\a^*[Z]}_T - Y^{Y_0, Z}_T\right)\Big].
\eeaa 
This idea of representation is further applied to study the case where the Principal can hire multiple Agents e.g.\cite{EMP18}, \cite{EP19} or Agent facing multiple Principals e.g. \cite{MR18}, \cite{HRY19} using the formulation of mean field games (as for the mean field game we refer to the seminal paper \cite{LL07} and the recent books \cite{CD1} and \cite{CD2}). 

However, the method relies on an important hypothesis: the stochastic system is non-degenerate. More specifically, at the heart of the dynamic programming approach to Principal-Agent problem lies the BSDE representation of the contract as shown in the above example, which requires the non-degeneracy of the system. 

It is well known that degenerate stochastic systems appear naturally in partially observed systems when we replace the non-observable part of the system by the filtered process, i.e. the conditional law of the unobservable given the observable, see e.g. \cite{Bensoussan}. One of the simplest cases of partial observed system is system with parameter uncertainty. This is studied by Fernandes and Phelan \cite{FP00}, Williams \cite{Williams09b}, then by Prat and Jovanovic \cite{Prat+Jovanovic}, which inspires and motivates this work.  To illustrate the idea, consider the following process
\begin{equation}\label{eq:intro}
	\mathrm{d}B_t = (\mu + \b_t)\mathrm{d}t + \sigma\mathrm{d}W_t,
\end{equation}
where $ \b $ is the control of the Agent and $ \mu $ represents a unknown parameter of the system which is called the time-invariant productivity. The common priors on $ \mu $ are normal with mean $ m_0 $ and variance $ V_0 $. Using Bayes' formula and replace $ \mu $ by the posteriors, which depend on $ B_t $ and the cumulative effort $ A_t := \int_{0}^t \b_s\mathrm{d}s $, we get the following controlled system:
\beaa
	\mathrm{d}B_t &=& \left(\frac{V_t}{V_0}m_0 + \frac{V_t}{\sigma^2}(B_t - A_t) + \b_t\right)\mathrm{d}t + \mathrm{d}I_t,
\eeaa
where $ V_t = \frac{\sigma^2V_0}{\sigma^2 + V_0t} $ is the posterior variance and $ (I_t)_{t\geq0} $ is a $ \dbF^B $-adapted Brownian motion called the innovation process. The partially observed state process \eqref{eq:intro} becomes now a degenerate controlled system with a new degenerate process $ (A_t)_{t\geq0} $. Since the process $ A $ is not observable by the Principal, it cannot be contracted upon. Indeed, to compute the posteriors on $ \mu $ at time $ t $, one needs to know the Agent's control $ \b $ up to time $ t $. The main difficulty is to find a dynamic representation of the contracts but using only the observable part of the system, namely in the above case the process $ B $. 

In this paper we consider general degenerate controlled system in a weak formulation, allowing the drift and the volatility of the filtered process (i.e. degenerate part of the system, denoted throughout the paper by $ X $) be controlled by the Agent. We assume that the observable part of the system (i.e. non-degenerate part of the system, denoted throughout the paper by the process $ B $) is controlled by the Agent only via the drift. We give the first order condition of the Agent's optimal control when he is given a contract. The first order condition is described by a path-dependent FBSDE. Contrary to the FBSDE literature, by path-dependent we mean the coefficients of the FBSDE dependent on the path of the forward process. In the Principal-Agent problem, the dependency mainly comes from the Agent's cost function and the contract. The wellposedness of these kinds of FBSDE is of independent interest and is studied in the accompanying paper \cite{hu2019}. Our strategy is to solve the Principal's problem first over the set of contracts described by the FBSDE given by the first order condition of the Agent's problem. The contracts described by the FBSDE may not necessarily be implementable, i.e. the optimal control of the Agent's problem may not exist. The implementability is then checked by the sufficient condition of the optimality of the Agent's problem.

The main contribution of the paper is to provide a way to solve the Principal-Agent's problem when only part of the state variables can be contracted upon, which is crucial when the system is partially observed. The rest of the paper is organised as follow. In section 2 we state our moral hazard problem in degenerated systems and give the main results of the paper. The proofs of the results will be given in Section 3. In Section 4 we solve the Principal-Agent problem in partially observed linear systems with explicit optimal contract and extend the result to a specific mean-field interacting Agents case. 

\section{Principal-Agent Problem in Degenerate Systems}
\subsection{Preliminaries}
Let $\O:= C\big([0,T], \dbR^d \big)$ be the canonical space and $ B $ the canonical process. Denote by $ \dbF=(\cF^B_t)_{0\leq t\leq T} $ the associated filtration and $ \dbP $ the Wiener measure. Let $ A $ be a compact convex subset of some finite dimensional space. Let
\begin{equation*}
b:[0,T]\times\O\times\mathbb{R}^n\times A\to\dbR^d\text{, }b(\cdot,a)\text{ }\dbF\text{-optional for any }a\in A,
\end{equation*}
\begin{equation*}
\eta:[0,T]\times\O\times\mathbb{R}^n\times A\to\dbR^n\text{, }\eta(\cdot,a)\text{ }\dbF\text{-optional for any }a\in A,
\end{equation*}
\begin{equation*}
\sigma:[0,T]\times\O\times\dbR^n\times A\to\cM_{n,d}(\dbR)\text{, }\sigma(\cdot,a)\text{ }\dbF\text{-optional for any }a\in A,
\end{equation*}
where $ \cM_{n,d}(\dbR) $ denotes the set of $ n\times d $ matrices with real entries. Throughout this paper we shall be studying control problems in a weak formulation. Denote $ \cU $ the set of admissible controls taking values in $ A $. For any $ \a\in\cU $, under the standard global Lipschitz conditions of Assumption \ref{A1} below, let $ X $ be the unique strong solution of the following stochastic differential equation
\begin{equation}\label{eq:X}
\mathrm{d}X_{t} = \eta_t(X_{t},\alpha_t)\mathrm{d}t + \sigma_t(X_t,\a_t)\mathrm{d}B_t,
\end{equation}
for some given initial data $ X_0 $.

Then we can define $ \mathrm{d}\dbP^\a|_{\cF_T} :=  \cE^\a_T\mathrm{d}\dbP|_{\cF_T} $ where
\begin{equation}\label{eq:exponential}
\cE^\a_T := \exp\left( \int_{0}^{T}b_t(X_t,\alpha_t)\cdot\mathrm{d}B_t + \frac12\int_{0}^{T}|b_t(X_t,\alpha_t)|^2\mathrm{d}t \right).
\end{equation}
In the case where the above process is a positive martingale, by Girsanov theorem we can define $ W^\alpha_t := B_t - B_0 - \int_0^tb_s(X_s,\alpha_s)\mathrm{d}s $, which is a $ d $-dimensional Brownian motion under $ \dbP^\alpha $ and which means that the canonical process $ B $ satisfies the following stochastic differential equation
\begin{equation}\label{eq:B}
\mathrm{d}B_t = b_t(X_t,\a_t)\mathrm{d}t + \mathrm{d}W^\a_{t}.
\end{equation}
This defines the family of probability measures on $ \Omega $
\bea
\cP:=\Big\{\dbP^\a\in\cP(\Omega) : \q \a\in \cU \notag	\Big\},
\eea
where $ \cP(\Omega) $ is the set of all probability measures on $ \Omega $.

\subsection{Problem Formulation}
The Agent signs a contract, works for the Principal for a predetermined period $ T $ and receives $ \xi $ at time $ T $. A contract $ \xi $ is a $ \mathcal{F}_{T} $-measurable random variable, which represents the payment the Agent receives at time $ T $. Let $ c:[0,T]\times\Omega\times A\to\dbR $ be a measurable function representing the Agent's cost of effort at time $ t $. Let $ k^A: [0,T]\times\Omega\times A\to\mathbb{R} $ be a bounded $ \dbF $-optional function and define $ \cK^A_{T} := \exp\left(-\int_{0}^{T}k^A_{t}(\a_t)\mathrm{d}t\right) $ representing the Agent's discount factor. The Agent aims at choosing an optimal effort $ \a\in\cU $ to optimize his utility when given a contract $ \xi $ proposed by the Principal:
\begin{equation}\label{eq:agentprob}
V_{A}(\xi) :=\sup_{\a\in \cU} J_A(\a):= \sup_{\a\in \cU}\mathbb{E}^{\a}\left[\cK^A_T\xi - \int_{0}^{T}\cK^A_tc_{t}(\a_t)\mathrm{d}t\right],
\end{equation}
where $\dbE^\a $ denotes the expectation under the probability measure $\dbP^\a\in\cP$ defined previously. For any contract $ \xi $, denote $ \mathcal{M}^{*}(\xi) \subset \cU$ the set of optimal controls of the Agent's problem.

The Principal on the other hand takes benefits from the outcome of the controlled process and pays the Agent accordingly. Also, the Agent's participation is conditioned on having his expected return above his reservation utility $ R $, in other words: $ V_{A}(\xi)\geq R $. In all, the Principal choose a contract among the following set of contracts in order to optimize her own utility:
\begin{equation}\label{admissibleContract}
\Upxi = \left\{ \xi ~~\mbox{$\cF_T$-measurable}:~~ \sup_{\dbP\in \cP}\dbE^{\dbP}\left[ \xi^2 \right]<\infty\text{ and }V_{A}(\xi)\geq R  \right\}.
\end{equation}
We will call the contracts in $ \Upxi $ the implementable contracts. In the case where the Agent has more than one optimal control, we follow the standard convention that the Agent chooses the one that is the best for the Principal. Let $ k^{P}: [0,T]\times\Omega\to\mathbb{R} $ be a bounded $ \dbF $-optional function and define $ \cK^{P}_{T} := \exp\left(-\int_{0}^{T}k^{P}_{t}\mathrm{d}t\right) $, which represents the discount factor of the Principal. She aims at solving the following optimization
\begin{equation*}
V_{P} = \sup_{\xi\in\Upxi}\sup_{\a^{*}\in\mathcal{M}^{*}(\xi)}\mathbb{E}^{\a^{*}}[\cK^{P}_{T}U(B_T - \xi)],
\end{equation*}
where the function $ U:\dbR\to\dbR $ is a given non-decreasing and concave utility function and we use the convention $ \sup\emptyset=-\infty $.

\begin{rem}
	One of our goals is to generalize the setting of the paper \cite{Touzi} by Cvitanic, Possamaï and Touzi, in which they consider the following system
	\begin{equation*}
	X_t = X_0 + \int_{0}^{t}\sigma_r(X,\b_t)(\l_r(X,a_r)\mathrm{d}r + \mathrm{d}W_r)\text{,  }t\in[0,T].
	\end{equation*}
	The drift of the dynamic needs to be in the range of the volatility, which means the system cannot be degenerate.
\end{rem}

\subsection{Motivation}\label{Section2.3}
The main motivation of our formulation is to tackle the Principal-Agent problem in a partially observed linear system. Using the above notation, the canonical process $ B $ represent the observable process, observed by the Agent and the Principal at the same time. The degenerate part of the system described by the process $ X $ represents the filtered process.

More precisely, let $ \nu = (\a, \b)\in\cU $ be an admissible control. Let $ \hat X $ be the unique strong solution of the following linear stochastic differential equation representing the unobservable part of the system:
\begin{equation}\label{eq:Xlinear}
\mathrm{d}\hat{X}_{t} = (\eta(t)\hat{X}_{t} + \alpha_{t})\mathrm{d}t + \sigma(t)\mathrm{d}W_{t}\text{, } \hat{X}_{0} = \mu,
\end{equation} 
where $ \mu $ a unobservable random variable independent of $ W $ assumed to be Gaussian with mean and variance $ m_0 $ and $ V_0 $, respectively. In addition, we assume that the Principal and the Agent both know $ m_0 $ and $ V_0 $. The functions $ \eta $ and $ \sigma $ are deterministic. Furthermore, the observable part of the system is generated by the noisy signal
\begin{equation}\label{eq:Blinear}
\mathrm{d}B_{t} = (h(t)\hat{X}_{t} + \beta_{t})\mathrm{d}t + \mathrm{d}W^\nu_{t},
\end{equation}
where $ W^\nu $ is a Brownian motion independent of $ W $ and $ h $ a deterministic function.

Notice that the couple $ (\hat X,B) $ is Gaussian and therefore so is the conditional distribution $ \cL(X_t|B_t) $, which is characterized by its mean and variance. Denote $ X_t := \dbE^\nu\left[ \hat X_t| \cF^B_t \right] $ and $ V_t := \dbE^\nu\left[ (\hat X_t - X_t)^2|\cF^B_t \right] $ for any given admissible control $ \nu $. The dynamics of the process $ X $ and $ V $ are given by the Kalman-Bucy filter as shown in the following proposition, the proof can be found in \cite[Chapter 2]{Bensoussan}.

\begin{prop}Let $ \nu=(\alpha,\beta) $ be an admissible control process. We have the following filter system:
	\begin{numcases}{}
	\mathrm{d}X_t = \big( \eta(t)X_t + \alpha_t\big)\mathrm{d}t + h(t)V(t)\mathrm{d}I^\nu_t,
	& $ X_0=m_0 $, \label{filterProcess_partial} \\
	\mathrm{d}B_t = \big(h(t)X_t + \beta_t \big)\mathrm{d}t + \mathrm{d}I^\nu_t & $ B_0 = 0 $,\label{observeproc_partial}
	\end{numcases}
	where $V$ is the solution to the ODE: 
	\bea\label{eq:varianceproc_partial}
	\dot V(t) = 2\eta(t)V(t) - h^2(t)V^2(t) + \sigma(t),\q V(0) = V_{0}.
	\eea
	The process $ I^\nu $ defined by
	\begin{equation}\label{eq:innov}
		I^\nu_t := W_t + \int_{0}^{t}h(s)\big(\hat X_s - X_s\big)\mathrm{d}s = B_t - \int_0^t (h(s)X_s + \beta_s)\mathrm{d}s
	\end{equation}
	is a $ \mathbb{F}^B  $-Brownian motion under $ \mathbb{P}^\nu $, and is called the Innovation process.
\end{prop}

Clearly under the filtered system \eqref{filterProcess_partial}-\eqref{observeproc_partial}, everything is fully observable, but the system itself becomes degenerate. We aim at solving the following Principal-Agent problem with the convention $ \sup\emptyset=-\infty $:
\begin{equation}\label{pb:principal_partial}
V_{P} = \sup_{\xi\in\Upxi}\sup_{\beta^{*}\in\mathcal{M}^{*}(\xi)}\mathbb{E}^{\beta^{*}}[B_T - \xi],
\end{equation}
where $ \beta^*\in\mathcal{M}^{*}(\xi) $ is the optimal response of the Agent given the contract $ \xi $, namely
\begin{equation}\label{pb:agent_partial}
\beta^* = \argmax_{\b\in\cU}\mathbb{E}^{\b}\left[ \xi - \int_{0}^{T}c(t,\b_t)\mathrm{d}t\right].
\end{equation}

Section \ref{Section4.1} below provides an explicit solution for this problem by applying the subsequent results of the Section \ref{MainResult}.

\begin{rem}
	Our model is inspired and motivated by the model proposed in \cite{Prat+Jovanovic}, where the observable process is described by
	\begin{equation*}
	\mathrm{d}B_t = (\mu + \beta_t)\mathrm{d}t + \mathrm{d}W_t,
	\end{equation*} 
	which is actually a special case of our framework, when $ \eta = \sigma = \alpha =0 $. Another difference is that they consider an infinite horizon problem, where the contract is given in the form of continuous salary whereas in our framework, the contract consists in a lump-sum payment at the end of the contract.
\end{rem}

More generally speaking, one can consider the following partially observed system:
\beaa
\mathrm{d}B_t &=& b_t(\hat X_t,\b_t)\mathrm{d}t + \mathrm{d}W_t \\
\mathrm{d}\hat X_t &=& \eta_t(\hat X_t,\a_t)\mathrm{d}t + \sigma_t(\hat X_t,\a_t)\mathrm{d}\tilde{W}_t,
\eeaa
where $ B $ and $ \hat X $ represent as before the observable process and the unobservable process, respectively. One way to transform the system into an equivalent but fully observable system is to replace the unobservable part by its (unormalized) conditional law described by Zakai equation. In this paper, we shall mainly focus on the finite dimensional case.

\begin{rem}
	The moral hazard framework of Principal-Agent problem is also frequently called the second best. The first best corresponds to the case where the Principal and the Agent have the same information. It is typically assumed that the Principal dictates the Agent's actions. Mathematically, the problem becomes a stochastic control problem for a single individual - the Principal. It is well-known that in non-degenerate systems under some mild conditions, the first best coincides with the second best in terms of Principal's value when the Agent's criterion is risk neutral. However, in a partially observed system, the result no longer holds true in general, even in the risk-neutral case. Consider the following example where there is no observable process. The unobservable process follows a SDE:
	\beaa
		\mathrm{d}\hat X_t = \a_t\mathrm{d}t + \mathrm{d}W_t \q \text{with }\hat X_0\sim \cN(m_0, V_0)\text{ independent of }\dbF^W.
	\eeaa
	Clearly, the filtered process is given by $ X_t = m_0 + \int_{0}^{t}\a_s\mathrm{d}s $. The Principal aims at optimizing his utility function:
	\begin{equation*}
		V_{P} = \sup_{\xi\in\Upxi}\sup_{\a^{*}\in\mathcal{M}^{*}(\xi)}\mathbb{E}^{\a^{*}}[\hat X_T - \xi],
	\end{equation*}
	where $ \a^* $ is the optimal response of the Agent given the contract $ \xi $, namely
	\begin{equation*}
		\a^* = \argmax_{\a\in\cU}\mathbb{E}^{\a}\left[ \xi - \frac12\int_{0}^{T}\a^2_t\mathrm{d}t\right].
	\end{equation*}
	In the second best case, the Principal has no information and can therefore only offer constant contracts. Agent has no incentive to work and therefore $ \a\equiv0 $. Taking into account the participation constraint the optimal contract is a constant payment $ R $ at time $ T $ where $ R $ is Agent's reservation value. The Principal's value in this case is $ m_0 - R $. In the first best case, the Principal observes the Agent's action and therefore the filtered process can be contracted upon. For simplicity, let's consider only linear contracts, i.e. $ \xi = cX_T + d $. The Agent's optimal control is $ \a^*\equiv c $ with participation constraint $ \frac{Tc^2}{2} + m_0c \geq R - d $. As for the Principal, for any given linear contract $ \xi $ we have
	\beaa
		\mathbb{E}^{\a^{*}}[\hat X_T - \xi] = \dbE^{\a^*} \left[ \hat X_T - \frac12\int_0^T(\a^*_t)^2\mathrm{d}t \right] - \dbE^{\a^*}\left[ \xi - \frac12\int_0^T(\a^*_t)^2\mathrm{d}t \right] \leq  m_0 + \frac{T}{2} - R.
	\eeaa
	The inequality becomes equality when $ c=1 $ and $ d=R-m_0-\frac{T}2 $, in which case the Principal's value is increased by $ \frac{T}2 $ comparing to the moral hazard case.
\end{rem}

\subsection{Main Results}\label{MainResult}
For presentation simplicity, we assume that the Agent's discount factor $ \cK_t^A $ is equal to $ 1 $ for all $ t\in[0,T] $. All the main results can be generalised without much difficulty to the case with $ \cK^A_{t} = \exp\left(-\int_{0}^{t}k^A_{s}(\a_s)\mathrm{d}s\right) $ where $ k^A: [0,T]\times\Omega\times A\to\mathbb{R} $ is a bounded $ \dbF $-optional function. Throughout this paper, we shall impose the following conditions on these coefficients.

\begin{assum}\label{A1}The maps $ b,\sigma,\eta $ are $ \cC^2 $ in $ x $. Moreover, there exists a constant $ L>0 $ and a modulus of continuity $ \varpi:[0,\infty)\to[0,\infty) $ such that for $ \phi=b,\sigma,\eta,\partial_xb,\partial_x\sigma,\partial_x\eta $, $ \dbP $-almost surely
\begin{equation*}
		|\phi_t(x,a) - \phi_t(x',a')| \leq L|x-x'| + \varpi(|a-a'|)\text{ for all }t\in[0,T] , x,x'\in\dbR^n , a,a'\in A
\end{equation*}
and for $ \phi=b,\sigma,\eta $, $ \dbP $-almost surely
\begin{equation*}
		|\partial_{xx}\phi_t(x,a) - \partial_{xx}\phi_t(x',a')| \leq \varpi(|x-x'| + |a-a'|)\text{ for all }t\in[0,T] , x,x'\in\dbR^n , a,a'\in A.
\end{equation*}
		
\end{assum}

The following assumption ensures that all weak controls $ \dbP^\a\in\cP $ are equivalent and that high order error terms in the variational calculus can be ignored in the non-linear system case.

\begin{assum}\label{A2}
	The Dol\'eans-Dade exponential $ \cE^\a $ defined in \eqref{eq:exponential} is a positive martingale for all $ \a\in\cU $. Furthermore, if the controlled system is not linear, namely not in the form of \eqref{eq:Xlinear}-\eqref{eq:Blinear}, there exists $ p>2 $ such that $ \sup\limits_{\a\in\cU}\dbE[(\cE^\a_T)^p]<\infty $.
\end{assum}

\no We aim to solve the Agent's problem \eqref{eq:agentprob} by variational calculus. In the following, for $ p\geq1 $, $ \mathbb{P}^\a\in\mathcal{P} $ and $ T>0 $, denote
\begin{equation*}
\mathbb{H}^p_T(\mathbb{P}^\a) := \left\{ Z\q\mathbb{F}-adapted:\|Z\|_{\dbH^p_T(\dbP^{\a})} := \mathbb{E}^\a\left[\int_{0}^{T}|Z_t|^p\mathrm{d}t\right]^{\frac1p}<+\infty \right\}.
\end{equation*}

For $ (t,\omega,x,x',a,a',p,q,z)\in[0,T]\times\Omega\times(\dbR^n)^2\times A^2\times\dbR^n\times\cM_{n,d}(\dbR)\times\dbR^d $, define the Hamiltonian:
\begin{equation}\label{eq: hamiltonian}
H_t(x,x',a,a',p,q,z) := p\cdot(\eta_t(x,a) + \sigma_t(x,a)b_t(x',a')) + \Tr[q\sigma^\intercal_t(x,a)] + z\cdot b_t(x,a) - c_t(a).
\end{equation}
Implicitly, at time $ t $ the Hamiltonian may depend on the path of canonical process $ B_{\cdot\wedge t} $ up to time $ t $. In the following, we will use frequently the partial derivative $ \partial_aH $, $ \partial_xH $, which are derivatives with respect to $ x $ and $ a $ but not $ x' $, $ a' $.

\begin{thm}[Necessary Condition]\label{necessaryCondition}
	Assume that Assumption \ref{A1} and \ref{A2} hold true. For a contract $ \xi\in\Upxi $, let $ \a^*\in\mathcal{M}^*(\xi) $. Then the following FBSDE
	\begin{numcases}{}
	\mathrm{d}Y_{t} = c_{t}(\a^*_t)\mathrm{d}t + Z_{t}\mathrm{d}W^{\a^*}_{t} & $ Y_{T}= \xi $, \label{stateequation01}  \\
	\mathrm{d}P_{t} = -\partial_x H_t(X_t,X_t,\a^*_t,\a^*_t,P_t,Q_t,Z_t)\mathrm{d}t + Q_{t}\mathrm{d}W^{\a^*}_{t} & $ P_{T}=0 $,  \label{adjointproc} \\
	\mathrm{d}B_t = \partial_zH_t(X_t,X_t,\a^*_t,\a^*_t,P_t,Q_t,Z_t)\mathrm{d}t + \mathrm{d}W^{\a^*}_t & $ B_0 = 0 $, \\
	\mathrm{d}X_t = \partial_p H_t(X_t,X_t,\a^*_t,\a^*_t,P_t,Q_t,Z_t) \mathrm{d}t + \sigma_t(X_t,\a^*_t)\mathrm{d}W^{\a^*}_{t} & $ X_0=x_0 $, \label{stateequation02}
	\end{numcases}
	where $ \a^* $ verifies
	\bea
	\a^*_t = \argmax_{a\in A}H_t(X_t,X_t,a,\a^*_t,P_t,Q_t,Z_t)\text{ for all }t\in[0,T], \dbP^{\a^*}-a.s. \label{opitmalControlEquation}
	\eea
	has an $\dbF$-adapted solution, denoted by
	\begin{equation*}
	(Y^*,Z^*,P^*,Q^*,X^*) \in \mathbb{H}^2_T(\mathbb{P}^{\a^*})^5.
	\end{equation*}
\end{thm}

\begin{rem}
	\begin{enumerate}
		\item In the case where $ \eta=\sigma=0 $ and $ b $ does not depend on $ x $, our formulation is reduced to the framework of Principal-Agent problem studied in \cite{Touzi} in the uncontrolled diffusion case, where the state equations are not degenerate. The FBSDE given in Theorem \ref{necessaryCondition} reduces to the BSDE \eqref{stateequation01}, which is the same representation of the contract under the optimal probability using the dynamic programming approach. To see that, note that the Hamiltonian in this case becomes
		\begin{equation*}
		H_t(a,z) = zb_t(a) - c_t(a)
		\end{equation*}
		and together with the definition of the optimal control given in \eqref{opitmalControlEquation}, the equation \eqref{stateequation01} can be written as
		\beaa
		\mathrm{d}Y_t &=& H_t(\a^*_t,Z_t)\mathrm{d}t + Z_t\mathrm{d}B_t \\
		&=& \max_{a\in A}H_t(a,Z_t)\mathrm{d}t + Z_t\mathrm{d}B_t\text{, }Y_T = \xi,
		\eeaa
		which is the canonical representation of the contract $ \xi $ given in \cite[Definition 3.2]{Touzi}.
		
		\item Since we are working on the weak formulation of stochastic control, it is not surprising that the necessary condition we obtain here does not coincide with the standard stochastic Pontryagin's Maximum Principle (see e.g. \cite{Peng90, YongZhou1999}). Comparing to the standard result, the forward-backward system has an extra backward equation \eqref{stateequation01}. The process $Y$ in \eqref{stateequation01} indeed characterizes the value function. The adjoint process \eqref{adjointproc}  and the forward one \eqref{stateequation02} are similar to those appearing in the standard stochastic maximum principle. 
		
		\item The contract $\xi$ appearing in the terminal condition of \eqref{stateequation01} may be path-dependent. To the best of our knowledge, the well-posedness of such FBSDE  has not yet been studied in the literature. For further details, we refer the readers to the accompanying paper on the path-dependent FBSDEs \cite{hu2019}.
		
		\item Another way to see how to obtain the FBSDE \eqref{stateequation01}-\eqref{stateequation02} is to conduct formally the variational calculus by introducing a new state process $ L $ representing the change of measure, namely
		\begin{equation*}
		\mathrm{d}L_t = b_t(X_t,\a_t)\mathrm{d}B_t
		\end{equation*}
		and rewrite the Agent's problem as following:
		\begin{equation*}
		V_{A} = \sup_{\a\in \cU}\mathbb{E}\left[L_T\xi - \int_{0}^{T}L_tc_{t}(\a_t)\mathrm{d}t\right],
		\end{equation*}
		where the controlled system becomes
		\begin{numcases}{}
		\mathrm{d}X_t = \eta_t(X_t,\a_t)\mathrm{d}t + \sigma_t(X_t,\a_t)\mathrm{d}B_t & $ X_0 = x_0 $, \nonumber \\
		\mathrm{d}L_t = L_tb_t(X_t,\a_t)\mathrm{d}B_t & $ L_0=1 $. \nonumber
		\end{numcases} 
		
		The standard Hamiltonian associated with the above control problem is given by
		\begin{equation}\label{eq: hamiltonian1}
		\hat{H}_t(x,l,a,p,q,z) = \eta_t(x,a)\cdot p + \Tr[q\sigma^\intercal_t(x,a)] + l(zb_t(x,a) - c_t(a)).
		\end{equation}
		
		The adjoint equations associated with $ X $ and $ L $ are given by
		\begin{numcases}
		\mathrm{d}P_t = -\partial_x\hat{H}_t(X_t,L_t,\a_t,P_t,Q_t,Z_t)\mathrm{d}t + Q_t\mathrm{d}B_t & $ P_T=0 $, \label{adjointproc1}\\
		\mathrm{d}Y_t = -\partial_l\hat{H}_t(X_t,L_t,\a_t,P_t,Q_t,Z_t)\mathrm{d}t + Z_t\mathrm{d}B_t & $ Y_T=\xi $ \label{stateequation11}.
		\end{numcases}
		
		The process $ (Y,Z) $ in \eqref{stateequation11} is actually the equation \eqref{stateequation01} and if we define $ \hat{P_t} = \frac{P_t}{L_t} $ and $ \hat{Q_t} = \frac{Q_t}{L_t} - \frac{P_tb_t(X_t,\a_t)}{L_t} $, then the process $ (\hat{P}, \hat{Q}) $ satisfies exactly the equation \eqref{adjointproc}. 
		
		However, there are a few problems in this approach. First of all, since the coefficients of the control system \eqref{adjointproc1}-\eqref{stateequation11} are not Lipschitz, to our best knowledge, one cannot directly apply the stochastic maximum principle. Our method provides a way to overcome this problem. Secondly, the newly introduced state variable $ L $ is not observable, the contract and the optimal control of the Agent should not depend on $ L $. And finally, note that the convexity of the Hamiltonian \eqref{eq: hamiltonian1} can be satisfied in very few situations due to the term $ lzb_t(x,a) $, which makes it tricky to get a useful sufficient condition for the Agent's problem. 
	\end{enumerate}
\end{rem}

Now we give a sufficient condition for the Agent's problem. Define the functional
\beaa
\cG(t,\o,x,x',a,a',p,q,z) &:=& H_t(x,x',a,a',p,q,z) + (x - x')qb_t(x,a) \\ & & + p\Big(\sigma_t(x,a)b_t(x,a) - \sigma_t(x,a)b_t(x',a')
- \sigma_t(x',a')b_t(x,a)\Big) .
\eeaa

\begin{thm}[Sufficient Condition]\label{sufficientCondition}
	Let Assumptions \ref{A1} and \ref{A2} hold true. Let $ (Y^*,Z^*,P^*,Q^*,X^*) $ be a solution to the system \eqref{stateequation01}-\eqref{stateequation02}, where $ \a^* $  satisfies \eqref{opitmalControlEquation}. If $ (x,a)\in\dbR^n\times A\mapsto \cG_t(x,X^*_t,a,\a^*_t,P^*_t,Q^*_t,Z^*_t) $ is concave, $ \alpha^* $ is an optimal control for the problem $ V_A(Y_T^*) $.
\end{thm}

We next focus on the Principal's problem. Based on the necessary and the sufficient condition of the Agent's problem, we consider the following sets of contracts:
\begin{equation*}
\ol{\Upxi} := \left\{  Y_T | Y_0\geq R, Z\in\ol\cZ \right\}\q\mbox{and}\q
\ul{\Upxi} := \left\{  Y_T | Y_0\geq R, Z\in\ul\cZ \right\},
\end{equation*}
where $ Y $ is defined by the equation \eqref{stateequation01} and 
\beaa
&\ol\cZ:=\Big\{Z\in\dbH^2_T(\dbP) :\exists \dbP^\a\in\mathcal{P}~s.t. ~~\mbox{\eqref{adjointproc},\eqref{stateequation02} have an $\dbF$-adapted solution $ (P,Q,X) $ in~$ \mathbb{H}^2(\dbP^\a)^3 $}\Big\},&\\
&\ul\cZ:=\Big\{Z\in \ol\cZ : ~~\mbox{$Q$ satisfies the sufficient condition given in Theorem \ref{sufficientCondition}}\Big\}.&
\eeaa
Clearly we have $\ul\Upxi\subset \Upxi\subset\ol{\Upxi} $. 
Further, we introduce the enlarged optimization of the Principal.
\begin{equation}
\bar{V}_P :=\sup_{Y_0\geq R}\sup_{Z\in\ol\cZ} \bar J_P(Y_0,Z) := \sup_{Y_0\geq R}\sup_{Z\in\ol\cZ}\mathbb{E}^{\a^{*}}\left[\cK^{P}_{T}U(B_T - Y_T)\right].
\end{equation}

\begin{thm}\label{principalProblem}
	Assume that Assumption \ref{A1} and \ref{A2} hold true. Then $ V_P\leq \bar{V}_P $. Moreover, if the control problem $ \bar{V}_P $ has a solution $ Z^* $ which satisfies the sufficient condition given in Theorem \ref{sufficientCondition}, then $ V_P = \bar{V}_P $ and $ \xi = Y^*_T $ is an optimal contract for $ V_P $.
\end{thm}

\section{Proofs}
Before giving the proof of the necessary condition, we give the following martingale representation theorem which we shall need in a moment.
\begin{lem}[Martingale Representation]\label{EMRT}
	Let $ (M_{t})_{0\leq t\leq T} $ be a stochastic process with decomposition
	\beaa
	\mathrm{d}M_{t} = h_{t}\mathrm{d}t + \mathrm{d}W_{t},
	\eeaa
	where $ h $ is a $ \mathbb{F}^{M} $-adapted process such that $ \dbE\left[\exp\left(\int^{T}_{0}h_{t}\mathrm{d}W_t - \frac12\int_{0}^{T}h_t^2\mathrm{d}t\right)\right]=1 $.
	Then for any $ \mathbb{F}^{M} $ martingale $ \xi=(\xi_{t})_{0\leq t\leq T} $, there exists a $ \mathbb{F}^{M} $-adapted process $ f $ such that 
	\beaa
	\forall t\in[0,T]\text{, }\xi_{t} = \xi_{0} + \int_{0}^{t}f_{s}\mathrm{d}W_{t}\text{,  }\mathbb{P}-a.s.
	\eeaa
\end{lem}

\begin{proof}
	Let $ (\xi_{t})_{0\leq t\leq T} $ be a $ \dbF^{M} $-martingale. For $ t\in[0,T] $, define
	\begin{equation*}
	L_{t} := e^{-\int_{0}^{t}h_{s}\mathrm{d}W_{s} - \frac{1}{2}\int_{0}^{t}h^{2}_{s}\mathrm{d}s}.
	\end{equation*}
	Clearly $  (L_{t})_{0\leq t\leq T} $ is a positive uniformly integrable martingale. By Girsanov's theorem, denote $ \mathbb{Q} $ the probability under which $ M $ is a Brownian motion, we have $ \left.\frac{\mathrm{d}\mathbb{Q}}{\mathrm{d}\mathbb{P}}\right\vert_{\mathcal{F}^M_{t}} = L_{t} $. By Bayes' rule, $ \frac{\xi}{L} $ is  a $ \mathbb{F}^{M} $ martingale under $ \mathbb{Q} $. Since $ M $ is now a Brownian motion under the probability $ \mathbb{Q} $, by the martingale representation theorem, there exists a $ \mathbb{F}^{M} $-adapted process $ k = (k_{t}, \mathbb{F}^{M}_{t}) $, $ 0\leq t\leq T $, such that $ \forall t\in[0,T] $, $ \frac{\xi_{t}}{L_{t}} = \xi_{0} + \int_{0}^{t}k_{s}\mathrm{d}M_{s} $.
	Finally, by Ito's formula, we have
	\begin{equation*}
	\mathrm{d}\xi_{t} = \mathrm{d}(L_{t}\cdot\frac{\xi_{t}}{L_{t}}) 
	= -\frac{\xi_{t}}{L_{t}}L_{t}h_{t}\mathrm{d}W_{t} + L_{t}k_{t}\mathrm{d}M_{t} - L_{t}h_{t}k_{t}\mathrm{d}t 
	= (L_{t}k_{t} - \xi_{t}h_{t})\mathrm{d}W_{t}.
	\end{equation*}
	\qed
\end{proof}

The following lemma gives an estimate on the error of the first order approximation of the difference between the solution of \eqref{eq:X} given two different control $ \a $ and $ \a' $.
\begin{lem}\label{lem: moment}
	Let $ \a $, $ \a' $ be two control processes and $ X $, $ X' $ be the corresponding solution of \eqref{eq:X}. Denote $ \Delta\a := \a - \a' $ and $ \Delta X := X - X' $. Let $ \Delta\hat{X} $ be the solution of the following linear SDE with initial condition $ \Delta\hat{X}_0 = 0 $:
	\begin{equation}\label{eq: linear SDE}
	\mathrm{d}\Delta\hat{X}_t = \left( \partial_x \eta_t(X'_t, \a'_t)\Delta\hat{X}_t + \partial_a \eta_t(X'_t,\a'_t)\Delta\a_t \right)\mathrm{d}t + \left( \partial_x \sigma_t(X'_t, \a'_t)\Delta\hat{X}_t + \partial_a \sigma_t(X'_t,\a'_t)\Delta\a_t \right)\mathrm{d}B_t.
	\end{equation}
	Then for any $ p\geq2 $, there exists $ K>0 $ depending only on $ p $, $ T $ and the coefficients $ \eta,\sigma $ of the equation such that 
	\begin{equation*}
	\mathbb{E}\left[ \sup_{t\in[0,T]}|\Delta\hat{X}_t - \Delta X_t|^{p} \right]\leq K\mathbb{E}\left[ \int_{0}^{T}|\Delta\a_t|^{2p}\mathrm{d}t \right],
	\end{equation*}
	where the expectation is taken under the Wiener measure.
\end{lem}

\begin{proof}
	In this proof, for the simplicity of the notation, we shall always use $ C $ to denote some strictly positive constants, which may vary from line to line. By Assumption \ref{A1} and Taylor expansion, we note that
	\begin{equation*}
	|\eta_t(X_t,\a_t) - \eta_t(X'_t,\a'_t) - \partial_x \eta_t(X'_t, \a'_t)\Delta X_t - \partial_a \eta_t(X'_t,\a'_t)\Delta\a_t| \leq L(|\Delta\a_t|^2 + |\Delta X_t|^2) ,
	\end{equation*}
	\begin{equation*}
	|\sigma_t(X_t,\a_t) - \sigma_t(X'_t,\a'_t) - \partial_x \sigma_t(X'_t, \a'_t)\Delta X_t - \partial_a \sigma_t(X'_t,\a'_t)\Delta\a_t| \leq L(|\Delta\a_t|^2 + |\Delta X_t|^2).
	\end{equation*}
	
	Therefore, by triangle inequality, we get 
	\begin{equation*}
	|\eta_t(X_t,\a_t) - \eta_t(X'_t,\a'_t) - \partial_x \eta_t(X'_t, \a'_t)\Delta \hat{X}_t - \partial_a \eta_t(X'_t,\a'_t)\Delta\a_t| \leq L(|\Delta\a_t|^2 + |\Delta X_t|^2 + |\Delta\hat{X}_t - \Delta X_t|),
	\end{equation*}
	\begin{equation*}
	|\sigma_t(X_t,\a_t) - \sigma_t(X'_t,\a'_t) - \partial_x \sigma_t(X'_t, \a'_t)\Delta \hat{X}_t - \partial_a \sigma_t(X'_t,\a'_t)\Delta\a_t| \leq L(|\Delta\a_t|^2 + |\Delta X_t|^2 + |\Delta\hat{X}_t - \Delta X_t|),
	\end{equation*}
	for some other constant $ L>0 $.
	
	By Burkholder-Davis-Gundy inequality and Jensen inequality, we have 
	\beaa
	\mathbb{E}\left[\sup_{t\in[0,T]}|\Delta\hat{X}_t - \Delta X_t|^p\right] 
	&\leq& C\mathbb{E}\Bigg[\int_{0}^{T} \left(|\Delta\a_t|^{2p} + |\Delta X_t|^{2p} + \sup_{t\in[0,T]}|\Delta\hat{X}_t - \Delta X_t|^{p}\right)\mathrm{d}t \\
	& & + \left(\int_{0}^{T}(|\Delta\a_t|^2 + |\Delta X_t|^2 + |\Delta\hat{X}_t - \Delta X_t|)^2\mathrm{d}t\right)^{p/2}\Bigg] \\
	&\leq& C\mathbb{E}\Bigg[\int_{0}^{T} \left(|\Delta\a_t|^{2p} + |\Delta X_t|^{2p} + \sup_{t\in[0,T]}|\Delta\hat{X}_t - \Delta X_t|^{p}\right)\mathrm{d}t\Bigg].
	\eeaa
	By Grönwall inequality, we get
	\begin{equation}\label{eq: gronwall1}
	\mathbb{E}\left[\sup_{t\in[0,T]}|\Delta\hat{X}_t - \Delta X_t|^p\right] \leq C\exp\left( CT \right)\mathbb{E}\Bigg[\int_{0}^{T} \left(|\Delta\a_t|^{2p} + |\Delta X_t|^{2p}\right)\mathrm{d}t\Bigg].
	\end{equation}
	Now we shall estimate the moments of $ \Delta X $. Again, for all $ p\geq2 $, by Burkholder-Davis-Gundy inequality and Jensen inequality, we have
	\beaa
	\mathbb{E}\left[\sup_{t\in[0,T]}|\Delta X_t|^p\right] 
	&\leq& C\mathbb{E}\Bigg[\int_{0}^{T} \left(|\Delta\a_t|^{p} + |\Delta X_t|^{p}\right)\mathrm{d}t + \left(\int_{0}^{T}(|\Delta\a_t| + |\Delta X_t|)^2\mathrm{d}t\right)^{p/2}\Bigg] \\
	&\leq& C\mathbb{E}\Bigg[\int_{0}^{T} \left(|\Delta\a_t|^{p} + \sup_{t\in[0,T]}|\Delta X_t|^{p}\right)\mathrm{d}t\Bigg].
	\eeaa
	
	By Grönwall inequality, we get
	\begin{equation}\label{eq: gronwall2}
	\mathbb{E}\left[\sup_{t\in[0,T]}|\Delta X_t|^p\right] \leq C\exp\left( CT \right)\mathbb{E}\Bigg[\int_{0}^{T}|\Delta\a_t|^{p}\mathrm{d}t\Bigg].
	\end{equation}
	
	Finally, combing \eqref{eq: gronwall1} and \eqref{eq: gronwall2}, we get
	\begin{equation*}
	\mathbb{E}\left[\sup_{t\in[0,T]}|\Delta\hat{X}_t - \Delta X_t|^p\right] \leq C\mathbb{E}\Bigg[\int_{0}^{T}|\Delta\a_t|^{2p}\mathrm{d}t\Bigg],
	\end{equation*}
	where $ C $ is a constant depending on $ p $, $ T $ and the coefficients of the SDE \eqref{eq:X}.
	\qed
\end{proof}

We shall now give the proof of Theorem \ref{necessaryCondition} in the non-linear system case. Note that in the linear case, $ \Delta X $ and $ \Delta \hat X $ defined in the above lemma are the same, which means that there will not be any high order error terms in the variational calculus. Therefore, the estimation of the error terms in step 5 of the following proof will not be necessary for linear system.

\vspace{5mm}

\no {\bf Proof of Theorem \ref{necessaryCondition}} \q
	For simplicity, we shall only prove the case where $ n=d=1 $. In higher dimensional case, the proof is essentially the same. Let  $ \xi $ be an implementable contract such that there exists an $ \mathbb{F} $-adapted control $ \a^{*} $ which optimizes the Agent's expected value. We denote the corresponding state variable $ X^* $.
	
	\ms
	\no {\it Step 1}.\q Introduce the dynamic version of the value function $ Y_{t} := \mathbb{E}^{\a^*}_t\left[ \xi - \int_{t}^{T} c_{s}(\a^*_s)\mathrm{d}s\right] $. One can check straightforward that $(M_t)_{0\leq t\leq T}:= \left(Y_{t}-\int_{0}^{t} c_{s}(\a^*_s)\mathrm{d}s\right)_{0\leq t\leq T}$
	is an $ \mathbb{F} $-martingale under the probability $ \mathbb{P}^{\a^{*}} $.  By the  martingale representation in Lemma \ref{EMRT}, there exists an $ \mathbb{F} $-adapted process $ Z $, such that for all $ t\in[0,T] $
	\begin{equation}\label{Mdecomp}
	M_t = J_A(\a^*) + \int_{0}^{t}Z_{s}\mathrm{d}W^{\a^{*}}_{s}.
	\end{equation}
	
	By Assumption \ref{A1} and \cite[Theorem 9.3.5]{CZ12}, $ Z\in \dbH^2_T(\dbP^{\a^*}) $.
	
	\no {\it Step 2}. Next we perform a variational calculus around the optimal control $\a^*$. Let $ \a \in 
	\cU$  and for all $ t\in[0,T] $ denote $ \Delta\alpha_t = \a_t - \a^*_t $. Define $ \a^\epsilon_t := \alpha^{*}_{t} + \epsilon\Delta\alpha_{t} $ and denote the corresponding state variable by $ X^{\epsilon} $. Since the control set $ \mathcal{U} $ is assumed to be convex, we have $\a^\epsilon \in \cU$. Denote $ \delta\a_t = \eps\Delta\a_t $, $ \Delta X_{t} = \frac{X^{\eps}_t - X^*_t}{\epsilon} $ and $ \delta X_t = \eps\Delta X_t $. As $ \xi = Y_T $, we have
	\bea
	J_A(\a^\epsilon) &=& \mathbb{E}^{\a^\eps}\left[ \xi - \int_{0}^{T} c_{t}(\a^\eps_t)\mathrm{d}t\right] \nonumber\\
	&=& \mathbb{E}^{\a^\epsilon}\left[J_A(\a^*) + \int_{0}^{T} c_{t}(\a^*_t)\mathrm{d}t + \int_{0}^{T}Z_t\mathrm{d}W^{\a^*}_t - \int_{0}^{T} c_{t}(\a^\epsilon_t)\mathrm{d}t \right] \nonumber\\
	&=& \mathbb{E}^{\a^\epsilon}\left[J_A(\a^*) + \int_{0}^{T}- \delta c_{t}\mathrm{d}t + \int_{0}^{T}Z_t\big(\mathrm{d}W^{\a^\eps}_t + (b_t(X^{\eps}_t,\a^{\eps}_t) - b_t(X^*_t,\a^*_t))\mathrm{d}t\big) \right] \nonumber
	\eea	
	where $ \delta c_t := c_t(\a^\epsilon_t) - c_t(\a^*_t) $. By the definition of the set of admissible contracts \eqref{admissibleContract}, we know that $ \int_{0}^{\cdot}Z_t\mathrm{d}W^{\a^\eps}_t $ is a true martingale under $ \dbP^{\a^\eps} $. Denote $ \cE^{\eps} $ the Radon-Nikodym derivative between $ \dbP^{\a^\eps} $ and $ \dbP^{\a^*} $, namely 
	\begin{equation*}
	\cE^{\eps} := \exp\left(\int_{0}^{T}(b_t(X^\eps_t,\a^\eps_t) - b_t(X^*_t,\a^*_t))\mathrm{d}W^{\a^*}_t - \frac12\int_{0}^{T}(b_t(X^\eps_t,\a^\eps_t) - b_t(X^*_t,\a^*_t))^2\mathrm{d}t\right).
	\end{equation*}
	Note that since $ b $ is bounded, $ \cE^\eps $ is in $ L^p $ for all $ p\geq 1 $. Applying Girsanov theorem and together with Taylor expansion on $ b $, we have
	\beaa
	&  & J_A(\a^\epsilon) - J_A(\a^*) \\
	&= & \mathbb{E}^{\a^*}\left[\cE^\eps\int_{0}^{T}\big(- \delta c_{t} + Z_t(b_t(X^{\eps}_t,\a^{\eps}_t) - b_t(X^*_t,\a^*_t))\big)\mathrm{d}t\right]  \\
	&= & \eps\mathbb{E}^{\a^*}\left[ \cE^\eps\int_{0}^{T}\big(- \partial_a c_{t}(\a^*_t)\Delta \a_t + Z_t(\partial_a b_t(X^{*}_t,\a^{*}_t)\Delta\a_t + \partial_xb_t(X^*_t,\a^*_t)\Delta \hat{X}_t) + \cR^\eps_t/\eps\big)\mathrm{d}t \right],
	\eeaa
	where $ \Delta \hat{X} $ is the solution of the SDE \eqref{eq: linear SDE} with initial condition $ \Delta \hat{X}_0 = 0 $ and 
	\bea
	\cR^\eps_t 
	&:=& - \Big(\delta c_{t}-\partial_a c_{t}(\a^*_t)\delta\alpha_t\Big) + Z_t\Big(b_t(X^{\eps}_t,\a^{\eps}) - b_t(X^*_t,\a^*_t) - \partial_a b_t(X^{*}_t,\a^{*}_t)\delta\a_t + \partial_xb_t(X^*_t,\a^*_t)\delta X_t\Big) \nonumber  \\
	&  & + \eps Z_t\partial_xb_t(X^*_t,\a^*_t)(\Delta X_t - \Delta\hat{X}_t). \label{eq:Reps} 
	\eea 
	
	We claim and will prove in {\it Step 5} that $\limsup_{\eps\rightarrow 0} \frac{1}{\eps^2}\dbE^{\a^*}\left[\cE^\eps\int_{0}^{T}|\cR^\eps_t|\mathrm{d}t\right]<\infty$. Then, 
	\begin{equation}\label{EqFiniteDiff}
	\frac{J_A(\a^\epsilon) - J_A(\a^*)}{\eps} = \mathbb{E}^{\a^*}\left[\cE^\eps\int_{0}^{T}\big(- \partial_a c_{t}(\a^*_t)\Delta \a_t + Z_t(\partial_a b_t(X^{*}_t,\a^{*}_t)\Delta\a_t + \partial_xb_t(X^*_t,\a^*_t)\Delta \hat{X}_t)\big)\mathrm{d}t\right] + \mathcal{O}(\eps).
	\end{equation}
	
	\ms
	
	\no{\it Step 3}. In order to continue the variational calculus above, we need to introduce the adjoint process $P$ and $Q$ satisfying \eqref{adjointproc}. Let $ (P,Q) $ be the unique strong solution in $ \dbH^2(\dbP^{\a^*}) $of the following BSDE:
	\begin{equation*}
	\mathrm{d}P_t = - \partial_xH_t(X^*_t,X^*_t,\a^*_t,\a^*_t,P_t,Q_t,Z_t)\mathrm{d}t + Q_t\mathrm{d}W^{\a^*}_t\text{, }P_T=0,
	\end{equation*}
	where $ H $ is the Hamiltonian given in \eqref{eq: hamiltonian}. Note that the above BSDE is a linear BSDE and since $ P_T=0 $, $ P $ and $ Q $ are in $ \dbH^p(\dbP^{\a^*}) $ for all $ p\geq2 $.

	\ms
	
	\no{\it Step 4}. \q We are going to continue the variational calculus above and obtain \eqref{opitmalControlEquation}. By It\^o's formula,	we have
	\beaa
	\mathrm{d}(P_t\Delta\hat{X}_t) & = &\Big(-Z_t\partial_xb_t(X^*_t,\a^*_t)\Delta\hat{X}_t + \Delta\alpha_t\Big(P_t(\partial_ab_t + b_t\partial_a\sigma_t)(X^*_t,\a^*_t) + Q_t\partial_a\sigma_t(X^*_t,\a^*_t)\Big)\Big)\mathrm{d}t \\ 
	& & + \Big( Q_t\Delta \hat{X}_t + P_t(\partial_x\sigma_t(X^*_t,\a^*_t)\Delta \hat{X}_t + \partial_a\sigma_t(X^*_t,\a^*_t)\Delta \a_t) \Big)\mathrm{d}W^{\a^*}_t.
	\eeaa
	
	\no Using the fact that $ P_{T} = 0 $ and $ \Delta\hat{X}_{0} = 0 $, we get
	\bea
	0 & = &\int_{0}^{T}\Big(-Z_tb_t(X^*_t,\a^*_t)\Delta\hat{X}_t + \Delta\alpha_t\Big(P_t(\partial_ab_t + b_t\partial_a\sigma_t)(X^*_t,\a^*_t) + Q_t\partial_a\sigma_t(X^*_t,\a^*_t)\Big)\Big)\mathrm{d}t \nonumber \\ 
	& & + \int_{0}^{T}\Big( Q_t\Delta \hat{X}_t + P_t(\partial_x\sigma_t(X^*_t,\a^*_t)\Delta \hat{X}_t + \partial_a\sigma_t(X^*_t,\a^*_t)\Delta \a_t) \Big)\mathrm{d}W^{\a^*}_t. \label{eq: adjoint_calculus01}
	\eea
	
	Note that
	\bea
	& & \int_{0}^{T}\Big( Q_t\Delta \hat{X}_t + P_t(\partial_x\sigma_t(X^*_t,\a^*_t)\Delta \hat{X}_t + \partial_a\sigma_t(X^*_t,\a^*_t)\Delta \a_t) \Big)\mathrm{d}W^{\a^*}_t \nonumber \\
	&=& \int_{0}^{T}\Big( Q_t\Delta \hat{X}_t + P_t(\partial_x\sigma_t(X^*_t,\a^*_t)\Delta \hat{X}_t + \partial_a\sigma_t(X^*_t,\a^*_t)\Delta \a_t) \Big)(b_t(X^\eps_t,\a^\eps_t) - b_t(X^*_t,\a^*_t))\mathrm{d}t \nonumber \\
	& & \qquad + \int_{0}^{T}\Big( Q_t\Delta \hat{X}_t + P_t(\partial_x\sigma_t(X^*_t,\a^*_t)\Delta \hat{X}_t + \partial_a\sigma_t(X^*_t,\a^*_t)\Delta \a_t) \Big)\mathrm{d}W^{\a^\eps}_t. \label{eq: adjoint_calculus02}
	\eea
	Denote
	\begin{equation}\label{eq: epsilon}
	\hat{\cR}^\eps_t := \Big( Q_t\Delta \hat{X}_t + P_t(\partial_x\sigma_t(X^*_t,\a^*_t)\Delta \hat{X}_t + \partial_a\sigma_t(X^*_t,\a^*_t)\Delta \a_t) \Big)(b_t(X^\eps_t,\a^\eps_t) - b_t(X^*_t,\a^*_t)).
	\end{equation}
	We claim and will prove in {\it Step 5} that $\limsup_{\eps\rightarrow 0} \frac{1}{\eps}\dbE^{\a^*}\left[\cE^\eps\int_{0}^{T}|\hat{\cR}^\eps_t|\mathrm{d}t\right]<\infty$.
	
	Inserting the equation \eqref{eq: adjoint_calculus01} and equation \eqref{eq: adjoint_calculus02} into \eqref{EqFiniteDiff}, we get
	\beaa
	\frac{J_A(\a^\epsilon) - J_A(\a^*)}{\eps} 
	&=& \mathbb{E}^{\a^*}\Big[\cE^\eps\int_{0}^{T}\Big(- \partial_a c_{t}(\a^*_t)\Delta \a_t + Z_t\partial_a b_t(X^{*}_t,\a^{*}_t)\Delta\a_t \\
	& & + \Delta\a_t\big(P_t(\partial_ab_t + b_t\partial_a\sigma_t)(X^*_t,\a^*_t) + Q_t\partial_a\sigma_t(X^*_t,\a^*_t)\big) \Big)\mathrm{d}t\Big] + \mathcal{O}(\eps).
	\eeaa
	Now, let $ \eps $ goes to $ 0 $ and by dominated convergence theorem, we obtain
	\beaa
	0\geq \lim_{\epsilon\to0}\left(\frac{J_A(\a^\epsilon) - J_A(\a^*)}{\eps}\right)
	&=& \mathbb{E}^{\a^*}\Big[\int_{0}^{T}\Delta \a_t\Big(- \partial_a c_{t}(\a^*_t) + Z_t\partial_a b_t(X^{*}_t,\a^{*}_t) \\
	& & + P_t(\partial_ab_t + b_t\partial_a\sigma_t)(X^*_t,\a^*_t) + Q_t\partial_a\sigma_t(X^*_t,\a^*_t)\Big)\mathrm{d}t\Big] \\
	&=& \dbE^{\a^*}\left[ \int_{0}^{T}\partial_a H_t(X^*_t,X^*_t,\a^*_t,\a^*_t,P_t,Q_t,Z_t)\Delta\a_t\mathrm{d}t \right]
	\eeaa
	Since $ \Delta\alpha $ is arbitrary, if $ \a^* $ is an optimal control, \eqref{opitmalControlEquation} holds true.
	
	\ms
	
	\no{\it Step 5}.\q Finally we shall complete the proof by proving the claim we made on $ (\cR^\eps_t)_{t\in[0,T]} $ and $ (\hat{\cR}^\eps_t)_{t\in[0,T]} $. 
	The first term in \eqref{eq:Reps} is clearly uniformly bounded by $ C\eps^2 $ for some positive constant $ C $, since the controls are bounded.
	
	By Hölder's inequality, we have
	\beaa
	& & \mathbb{E}^{\a^*}\left[ \cE^\eps\int_{0}^{T}Z_t\Big(b_t(X^{\eps}_t,\a^{\eps}) - b_t(X^*_t,\a^*_t) - \partial_a b_t(X^{*}_t,\a^{*}_t)\delta\a_t + \partial_xb_t(X^*_t,\a^*_t)\delta X_t\Big)\mathrm{d}t \right] \\
	&\leq& \mathbb{E}^{\a^*}\left[ \cE^\eps\left(\int_{0}^{T}Z^2_t\mathrm{d}t\right)^{\frac12}\left(\int_{0}^{T}\Big(b_t(X^{\eps}_t,\a^{\eps}) - b_t(X^*_t,\a^*_t) - \partial_a b_t(X^{*}_t,\a^{*}_t)\delta\a_t + \partial_xb_t(X^*_t,\a^*_t)\delta X_t\Big)^2\mathrm{d}t\right)^{\frac12} \right] \\
	&\leq& \|\cE^\eps\|_{p}\|Z\|_{\dbH^2_T(\dbP^{\a^*})}\dbE^{\a^*}\left[\left(\int_{0}^{T}\Big(b_t(X^{\eps}_t,\a^{\eps}) - b_t(X^*_t,\a^*_t) - \partial_a b_t(X^{*}_t,\a^{*}_t)\delta\a_t + \partial_xb_t(X^*_t,\a^*_t)\delta X_t\Big)^2\mathrm{d}t\right)^{\frac{q}{2}} \right]^\frac1q \\
	&\leq& CT\sup_\eps\|\cE^\eps\|_{p}\|Z\|_{\dbH^2_T(\dbP^{\a^*})}\dbE^{\a^*}\left[\left(\int_{0}^{T}\Big(\Delta\a^2_t + \Delta X^2_t\Big)\mathrm{d}t\right)^{\frac{q}{2}} \right]^\frac1q\eps^2,
	\eeaa
	where $ \frac{1}{p} + \frac{1}{q} = \frac12 $ with $ p $ given in Assumption \ref{A2}. As shown in the proof of Lemma \ref{lem: moment}, $ \Delta X\in\dbH^p_T(\dbP^{\a^*}) $ for all $ p\geq2 $, therefore, the second term of the right hand side of \eqref{eq:Reps} is also bounded by $ C\eps^2 $ for some positive constant $ C $. As for the third term in \eqref{eq:Reps}, again by Lemma \ref{lem: moment}, using the definition $ \eps\Delta X_t = X^\eps_t - X^*_t $, we have
	\beaa
	\mathbb{E}^{\a^\eps}\left[ \int_{0}^{T}\eps |Z_t\partial_xb_t(X^*_t,\a^*_t)(\Delta X_t - \Delta\hat{X}_t)|\mathrm{d}t \right] 
	&\leq& C\mathbb{E}\left[ \hat{\cE}^{\eps}\int_{0}^{T}\eps|Z_t(\Delta X_t - \Delta\hat{X}_t)|\mathrm{d}t \right] \\
	&\leq& C\|\hat{\cE}^{\eps}\|_{p}\|Z\|_{\dbH^2_T(\dbP)}\dbE\left[ \sup_{t\in[0,T]}|\eps\Delta X_t - \eps\Delta\hat{X}_t|^q \right]^{\frac1q} \\
	&\leq& C\sup_\eps\|\hat{\cE}^{\eps}\|_{p}\|Z\|_{\dbH^2_T(\dbP)}\dbE\left[ \sup_{t\in[0,T]}|\Delta\a_t|^{2q} \right]^\frac1q\eps^{2},
	\eeaa
	where $ \frac{1}{p} + \frac{1}{q} = \frac12 $ with $ p $ given in Assumption \ref{A2} and $ \hat{\cE}^{\eps} $ is the Radon-Nikodym derivative between $ \dbP^{\a^\eps} $ and the Wiener measure $ \dbP $, namely
	\begin{equation*}
	\hat{\cE}^{\eps} := \exp\left(\int_{0}^{T}b_t(X^\eps_t,\a^\eps_t) \mathrm{d}B_t - \frac12\int_{0}^{T}b^2_t(X^\eps_t,\a^\eps_t)\mathrm{d}t\right),
	\end{equation*}
	which has finite moment uniformly in $ \eps $ by Assumption \ref{A2}.  Finally to complete our proof, we shall prove $\limsup_{\eps\rightarrow 0} \frac{1}{\eps}\dbE^{\a^*}\left[\cE^\eps\int_{0}^{T}|\hat{\cR}^\eps_t|\mathrm{d}t\right]<\infty$ where $ \hat{\cR}^\eps_t $ is defined by the equation \eqref{eq: epsilon}. Note that $ \a^\eps - \a^* = \eps\Delta \a $ and as shown in the proof of Lemma \ref{lem: moment},
	\begin{equation*}
	\dbE^{\a^*}\left[\int_{0}^{T}|X^\eps_t-X^*_t|^p\mathrm{d}t\right]\leq C\mathbb{E}\left[\sup_{t\in[0,T]}|X^\eps_t - X^*_t|^p\right] \leq C\mathbb{E}\Bigg[\int_{0}^{T}|\Delta\a_t|^{p}\mathrm{d}t\Bigg]\eps^{p}
	\end{equation*}
	for some constants $ C $. Therefore, by Assumption \ref{A1}, using Cauchy-Schwarz inequality and Hölder inequality, we have
	\beaa
	&&\dbE^{\a^*}\left[ \cE^\eps\int_{0}^{T}|\hat{\cR}^\eps_t|\mathrm{d}t \right] \\
	&\leq& C\dbE^{\a^*}\left[ \int_{0}^{T}\left(|P_t|^4 + |Q_t|^4 + |\Delta X_t|^4 + |\Delta \a_t|^4 \right)\mathrm{d}t \right]^{\frac12}\dbE^{\a^*}\left[ \left(\int_{0}^{T}(|\Delta X_t|^2 + |\Delta \a_t|^2)\mathrm{d}t\right)^2 \right]^{\frac14}\eps,
	\eeaa
	for some constant $ C $. Therefore, $\limsup_{\eps\rightarrow 0} \frac{1}{\eps}\dbE^{\a^*}\left[\cE^\eps\int_{0}^{T}|\hat{\cR}^\eps_t|\mathrm{d}t\right]<\infty$.
	\qed

\no {\bf Proof of Theorem \ref{sufficientCondition}} \q
Let $ \a\in \cU$  and denote the corresponding state variable by $ X $. Denote $ \Delta\a := \a - \a^* $ and $ \Delta X = X - X^* $. By the same computation as in the proof of Theorem \ref{necessaryCondition} and the concavity of $ \cG $, we have
\beaa
&  & J_A(\a) - J_A(\a^*) \\
&= & \mathbb{E}^{\a}\left[\int_{0}^{T}\big(- \delta c_{t} + Z_t(b_t(X_t,\a_t) - b_t(X^*_t,\a^*_t))\big)\mathrm{d}t + P_T\Delta X \right]  \\
&= & \mathbb{E}^{\a}\Bigg[ \int_{0}^{T}\big(\cG_t(X_t,X^*_t,\a_t,\a^*_t,P_t,Q_t,Z_t) - \cG_t(X_t,X^*_t,\a_t,\a^*_t,P_t,Q_t,Z_t) \\
& & \qquad\qquad\qquad\qquad\qquad - \partial_x\cG_t(X^*_t,X^*_t,\a^*_t,\a^*_t,P_t,Q_t,Z_t)\Delta X_t\big)\mathrm{d}t \Bigg]\leq 0.
\eeaa\qed

\no {\bf Proof of Theorem \ref{principalProblem}} \q
Since the enlarged Principal's problem has a larger admissible contract set, the Principal's value is higher than the initial problem of the Principal. If $ \bar V_P $ has an optimal solution $ \xi^* $ such that the solution $ Q^* $ of the FBSDE \eqref{stateequation01}-\eqref{stateequation02} satisfies the sufficient condition given by Theorem \ref{sufficientCondition}, then we know that the control $ \a^* $ defined by \eqref{opitmalControlEquation} is indeed the optimal control for the Agent's Problem given this contract $ \xi^* $. Therefore, $ \xi^*\in\Upxi $ and consequently $ V_P\geq \bar V_P $. Combined with the first statement of the theorem, we get $ V_P=\bar V_P $.
\qed

\section{Applications}
We shall make the following assumption throughout this section.
\begin{assum}\label{A3}
	The Agent's cost function $ (t,b)\mapsto c(t,b) $ is deterministic, strictly convex, increasing and twice differentiable in $ b $.
\end{assum}

\subsection{Principal-Agent Problem in Partially Observed Linear System}\label{Section4.1}
We return to the context in Section \ref{Section2.3} and aim at solving the Principal-Agent problem  \eqref{pb:principal_partial}-\eqref{pb:agent_partial}. In addition, we assume that the unobservable process $ \hat X $ is non-affected by the Agent's control $ \a $. We recall that the controlled state variables follow \eqref{filterProcess_partial}-\eqref{observeproc_partial} with $ \a\equiv0 $ and that the associated Hamiltonian function is given by
\begin{equation*}
H(t, x, b, p, q, z) := \Big(\eta(t)x - h(t)V(t)(h(t)x + b)\Big)p - \Big( h(t)x + b \Big)z + c(t,b).
\end{equation*}

The necessary condition by Theorem \ref{necessaryCondition} becomes
\begin{numcases}{}
\mathrm{d}Y_{t} = c(t,\beta^*_t)\mathrm{d}t + Z_{t}\mathrm{d}I_{t} & $ Y_{T}=\xi $, \label{eq:Yriskneutre} \\
\mathrm{d}P_{t} = \left(h(t)Z_{t} - (\eta(t) - V(t)h^2(t))P_t\right)\mathrm{d}t + Q_{t}\mathrm{d}I_{t} & $ P_{T}=0 $, \label{eq:adjointriskneutre} \\
\mathrm{d}X_t = \eta(t)X_t\mathrm{d}t + h(t)V(t)\mathrm{d}I_{t} & $ X_0=m_0 $, \label{eq:Xriskneture}
\end{numcases}
where 
\bea
\b^*_t = \argmin_{b\in A }H(t, X_t,b,P_t,Q_t,Z_t). \label{opitmalControlEquation_partial}
\eea
The above system describes the contracts under the optimal response of the Agent. The sufficient condition can also be re-written under simpler form:
\begin{thm}\label{sufficientCondition_partial}
	Let $ (Y^*,Z^*,P^*,Q^*,X^*) $ be a solution to the system \eqref{eq:Yriskneutre}-\eqref{eq:adjointriskneutre}, where $ \b^* $  satisfies \eqref{opitmalControlEquation_partial}. Define $\xi:=Y^*_T $. Then $ \b^* $ is an optimal control if 
	\begin{equation}\label{QsufficientCondition_partial}
	\frac{Q^*_t}{2h(t)}\in\left[0, \inf_{b\in A }\left\{\partial^2_{bb} c(t,b)\right\}\right]
	\q\mbox{for all}\q t\in[0,T] .
	\end{equation}
	In particular, in this case $\xi\in \Upxi$.
\end{thm}
\begin{rem}
	We can also obtain the necessary and sufficient condition in the case of exponential utility Agent, the proof and the hypothesis are slightly different, we shall detail them in the Appendix.
\end{rem}

Together with Theorem \ref{principalProblem}, we can solve the Principal's problem explicitly.

\begin{thm}
	Let Assumption \ref{A3} hold true. Assume in addition that the function $ b\in B\mapsto b - c(t,b) $ has a maximiser, denoted by $ \bar{\beta}^*(t) $. Define
	\begin{equation*}
	Z^{*}_{t} = \partial_bc(t,\bar{\beta}^{*}(t)) + \int_{t}^{T}e^{\int_{t}^{s}\eta(r)\mathrm{d}r}h(s)\partial_bc(s,\bar{\beta}^{*}(s))\mathrm{d}s. 
	\end{equation*}
	The optimal contract for the Principal's problem is given by
	\begin{equation}\label{eq:optimalcontractrn}
	\xi^* = R + \int_{0}^{T}(c(t,\bar{\beta}^*(t)) - Z^*_t(h(t)\bar{X}^*_t + \bar{\beta}^*(t)))\mathrm{d}t + \int_{0}^{T}Z^*_{t}\mathrm{d}B_t,	
	\end{equation}
	where $ \bar{X}^*_{t} $ is the unique solution to \eqref{eq:Xriskneture} with innovation process $ I $ computed with $ \bar \b^* $ and the Principal's optimal expected utility is
	\begin{equation*}
	V_{P} = -R + \int_{0}^{T}\left(\bar{\beta}^*(t) - c(t,\bar{\beta}^*(t))\right)\mathrm{d}t + \int_{0}^{T}h(t)m(t)\mathrm{d}t,
	\end{equation*}
	where $ m(t):= m_0e^{\int_{0}^{t}\eta(s)\mathrm{d}s} $ for $ t\in[0,T] $.
\end{thm}

\begin{rem}
	It is immediately verified that the optimal contract \eqref{eq:optimalcontractrn} is indeed implementable as it depends only on the observable of $ B $. In contrast with the classical Principal-Agent literature, the filtered process $ X $ is not observable by the Principal and cannot be contracted upon. In fact, in order to compute the filtered process $ X $, one needs to observe the Agent's control $ \b $ to compute the innovation process $ I $ given by \eqref{eq:innov} before solving the equation \eqref{eq:Xriskneture}.
\end{rem}

\begin{proof}
	Note that for any $Z\in \dbH^2_T$,
	\bea
	\bar J_P(Y_0,Z) &=&\mathbb{E}^{\beta^*}\left[\left(B_T - Y_T\right)\right]
	= \mathbb{E}^{\beta^{*}}\left[\left(\int_{0}^{T}(h(t)X_{t} + \beta^{*}_{t})\mathrm{d}t - \int_{0}^{T}c(t,\beta^*_t)\mathrm{d}t\right) \right] - Y_0 \notag \\
	&=& \mathbb{E}^{\beta^{*}}\left[\int_{0}^{T}\left(\beta^{*}_{t} - c(t,\beta^*_t)\right)\mathrm{d}t \right] - Y_0 + \int_{0}^{T}h(t)m(t)\mathrm{d}t,\notag \\
	&\le &  \mathbb{E}^{\beta^{*}}\left[ \int_{0}^{T}\left(\bar\beta^{*}_{t} - c(t,\bar\beta^*(t))\right)\mathrm{d}t \right] - R + \int_{0}^{T}h(t)m(t)\mathrm{d}t. \label{eq:barVupper}
	\eea
	We are going to show that $ \bar{\beta}^* $ is Agent's optimal control associated to a contract $\xi^*\in \ul\Upxi$.
	In order for $\bar\b^*$ to be an optimal control, it follows from \eqref{opitmalControlEquation_partial} that the processes $Z^*$ and $P^*$ in the forward-backward system must satisfy
	\begin{equation}\label{eq:ZandP}
	Z^*_{t} = \partial_bc(t,\bar{\beta}^{*}(t)) - V(h)h(t)P^*_{t}.
	\end{equation}
	Inserting the equality above into \eqref{eq:adjointriskneutre}, we obtain
	\begin{equation*}
	\mathrm{d}P_{t} = (h(t)\partial_bc(t,\bar{\beta}^{*}(t)) - \eta(t)P_{t})\mathrm{d}t + Q_{t}\mathrm{d}I^{\bar{\beta}^{*}},\q P_{T}=0.
	\end{equation*}
	This BSDE does have a unique solution given by
	\begin{equation*}
	P^{*}_{t} = -\int_{t}^{T}e^{\int_{t}^{s}\eta(r)\mathrm{d}r}h(s)\partial_bc(s,\bar{\beta}^{*}(s))\mathrm{d}s,
	\q Q^*_{t} = 0,\q  t\in[0,T].
	\end{equation*}
	Note that $Q$ satisfies \eqref{QsufficientCondition_partial}. Again by \eqref{eq:ZandP}, we have
	\begin{equation*}
	Z^{*}_{t} = \partial_bc(t,\bar{\beta}^{*}(t)) + \int_{t}^{T}e^{\int_{t}^{s}\eta(r)\mathrm{d}r}h(s)\partial_bc(s,\bar{\beta}^{*}(s))\mathrm{d}s, 
	\end{equation*}
	which is a deterministic process. Further, there is a unique $\dbF^B$-adapted solution $\bar X^*$ to the SDE \eqref{filterProcess_partial} with $ \a\equiv0 $. Using the process $Z^*$ and $\bar X^*$, we immediately find an $\dbF^B$-adapted process $Y^*$ satisfying \eqref{eq:Yriskneutre}. So far, we find an $\dbF^B$-adapted solution $ (Y^*,Z^*,P^*,Q^*,X^*) $ to the forward-backward system. 
	Let $ \xi^* $ be the contract defined by $\xi^*:= Y^*_T$, in other words, $\xi^*$ satisfies \eqref{eq:optimalcontractrn}. Clearly, we have $\xi^*\in \ul\Upxi$ and $\bar\b^*$ is the corresponding Agent's optimal control. Therefore, by \eqref{eq:barVupper} and $\ul\Upxi\subset \Upxi\subset\ol{\Upxi} $ we obtain $\bar V_P= V_P$. 
	\qed
\end{proof}

\subsection{Mean Field Interacting Agents in Partially Observed Linear System}
In our second application, we consider the Principal-Agent problem in the same partially observed system but with $ N $ interacting Agents.
As in Section \ref{Section2.3}, we denote $ \hat X $ as the unobservable process and $ B $ as the observable, both assumed to be 1-dimensional for simplicity. Agents $ i=1,\cdots,n $ have private state processes $ (\hat X^i,B^i) $ with dynamics
\beaa
\mathrm{d}B^i_{t} = (h(t)\bar{\lambda}^N_t \hat X^i_{t} + \beta^i_{t})\mathrm{d}t + \mathrm{d}W^{\beta^i}_{t},\q B_0 = 0,
\eeaa
where $\bar{\lambda}^N_t = \frac{1}{N}\sum_{1}^{N}B^i_t$ is the mean value of the observables at time $ t $ and $ \hat X $ is the unique strong solution of
\beaa
\mathrm{d}\hat X^i_{t} = \eta(t)\hat X^i_{t}\mathrm{d}t + \sigma(t)\mathrm{d}W^i_{t},\q \hat X_{0} = \mu^i.
\eeaa
Here, $ W^1,\cdots,W^N,W^{\b^1},\cdots,W^{\b^N} $ are independent Brownian motions, and $ \beta^i $ is the control chosen by the Agent $ i $. We assume that $ \mu^1,\cdots,\mu^N $ are independent Gaussian variables with mean $ m_0 $ and variance $ V_0 $.

The Principal proposes the same contract $ \xi $ to all the Agents. Then the interacting agents agree on a Nash equilibrium in which each Agent cannot improve his utility by deviating unilaterally.

The problem becomes very difficult in general because of the coupling. Note that the observable process $ B^i $ for Agent $ i $ is no longer Gaussian and the filter of $ X^i $, namely the conditional law of $ X^i $ given $ B^i $, is no longer described by a finite-dimensional process. 

However, as the number of Agents $ N\rightarrow\infty$, we expect the system to become decoupled, in which case we may be able to learn something about the Nash equilibrium from the corresponding mean field game (MFG).

\subsubsection{The Setting of Mean Field Interacting Agent Problem}
As before, denote $ \hat{X} $ the unique strong solution of the following linear stochastic differential equation, which represents the unobservable part of the system.
\begin{equation*}
\mathrm{d}\hat{X}_{t} = \eta(t)\hat{X}_{t} \mathrm{d}t + \sigma(t)\mathrm{d}W_{t},\q X_{0} = \mu.
\end{equation*}
Let $\l\in C([0,T],\dbR)$ represent the observed population mean, for any $ \b\in\cU $, define
\begin{equation*}
\left.\frac{\mathrm{d}\dbP^\b}{\mathrm{d}\dbP}\right|_{\cF_T} = \exp\left( \int_{0}^{T}(h(t)\l_t\hat{X}_t + \beta_t)\mathrm{d}B_t + \frac12\int_{0}^{T}(h(t)\l_t\hat{X}_t + \beta_t)^2\mathrm{d}t \right).
\end{equation*}

By Girsanov theorem, the canonical process $ B $ satisfies the following stochastic differential equation under $ \dbP^\b $:
\begin{equation}
\mathrm{d}B_t = (h(t) \l_t  \hat X_{t} + \beta_{t})\mathrm{d}t + \mathrm{d}W^{\beta}_{t},\q B_{0} = 0\label{eq:Omfg},
\end{equation}
where $ W^\b_t $ is a Brownian motion under $ \dbP^\b $. Each Agent receives the payment $ \xi $ at the maturity $ T $ of the contract and aims at optimizing his utility function:
\begin{equation*}
	V_{A}(\xi, \lambda) = \sup_{\dbP^\beta\in \cP(\l)}\mathbb{E}^{\beta}\left[\xi - \frac{1}{2}\int_{0}^{T}c(t,\b_t)\mathrm{d}t\right].
\end{equation*}

Denote $ \mathcal{M}^{*}(\xi,\lambda) $ the set of optimal controls of the Agent's problem given the contract $ \xi $ and the observed population mean $ \l $. Define  
\beaa
\Phi(\xi, \lambda) := \{ \tilde\l : \q \tilde\l_t = \dbE^{\b^*}[B_t]\q\mbox{for all}~t\in[0,T]\q\mbox{and}\q \b^*\in\mathcal{M}^{*}(\xi,\lambda)  \}.
\eeaa
A mean field equilibrium is a fixed point of the map $\l\mapsto \Phi(\xi,\l) $. Denote $ \Theta(\xi) $ the set of equilibria among the Agents when given the contract $ \xi $ satisfying the participation constraint, namely the Agent's value should be above the reservation utility $ R $. The Principal's optimization problem is given by
\begin{equation*}
V_{P} = \sup_{\xi\in\Upxi}\sup_{\lambda\in\Theta(\xi)}\sup_{\beta^{*}\in\mathcal{M}^{*}(\xi, \lambda)}\mathbb{E}^{\beta^{*}}[B_T - \xi]
\end{equation*}
with the convention $ \sup\emptyset=-\infty $.

\begin{rem}
	As discussed in Elie, Mastrolia, Possamaï \cite{Mastrolia}, we are only going to consider contracts such that the Principal is able to compute the reaction of the Agents. In other words, the set of contracts given which there is at least one Nash equilibrium among the Agents. Indeed, the Principal needs to be able to anticipate how the Agents are going to react to the contracts that he may offer, and will therefore not offer contracts for which Agents cannot agree on an equilibrium.
\end{rem}

\subsubsection{The Agent's Problem}
For any fixed $\l \in C([0,T],\dbR)$, the observable process $ B $ in \eqref{eq:Omfg} is a Gaussian process and the Agent can compute the conditional mean and variance of the unobservable process $ (\hat{X}_t)_{t\geq0} $ using the Kalman filter. Denote as before $ X_t := \mathbb{E}^{\b}[\hat X_t|\mathcal{F}^B_t] $ and $ V_t = \mathbb{E}^{\b}[(\hat{X}_t - X_t)^2|\mathcal{F}^B_t] $.

\begin{prop}\label{prop:necessaryMFG}
	Let $\l \in C([0,T],\dbR)$ be the observed population mean and $ \beta\in\cU $. We have the following filter system:
	\begin{numcases}
	\text{ }\mathrm{d}X_t = \eta(t)X_t\mathrm{d}t + h(t)V(t)\lambda_t\mathrm{d}I^\beta_t & $ X_0=m_0 $, \nonumber \\
	\mathrm{d}V_t = (2\eta(t)V(t) - h^2(t)V^2(t)\lambda_t^2 + \sigma(t))\mathrm{d}t & $ V_{0} = V_{0} $, \label{eq:filtreMFG} \\
	\mathrm{d}B_t = (h(t)\lambda_tX_{t} + \beta_{t})\mathrm{d}t + \mathrm{d}I^\beta_t & $ B_0 = 0 $, \nonumber
	\end{numcases}
	where $ I^\b $ is the innovation process.
\end{prop}
Using the filtered system, the Agent's problem becomes fully observable. We now recall the necessary and sufficient condition for the Agent's problem in this specific case in the two following propositions.

\begin{prop}\label{prop:sufficientMFG}
	Let $\l \in C([0,T],\dbR)$ be the observed population mean. For a contract $ \xi\in\Upxi $, let $ \beta^*\in\mathcal{M}^*(\lambda,\xi) $. Then the following FBSDE:
	\begin{numcases}{}
	\mathrm{d}Y_{t} = c(t,\b^*_t)\mathrm{d}t + Z_{t}\mathrm{d}I^{\beta^*}_{t} & $ Y_{T}=\xi $, \label{MFGY}\\
	\mathrm{d}P_{t} = \Big(h(t)\lambda_tZ_{t} - (\eta(t) - V(t)h^2(t)\lambda_t^2)P_t\Big)\mathrm{d}t + Q_{t}\mathrm{d}I^{\beta^*}_{t} & $ P_{T}=0  \notag$,\\  
	\mathrm{d}X_t = \eta(t)X_t\mathrm{d}t + h(t)V(t)\lambda_t\mathrm{d}I^{\beta^*}_{t} & $ X_0=m_0 $, \notag \\
	\mathrm{d}B_t =  (h(t)\lambda_tX_{t} + \beta^*_{t})\mathrm{d}t + \mathrm{d}I^{\beta^*}_t & $ B_0 = 0 $, \label{MFGO}
	\end{numcases}
	has a solution, denoted 
	\begin{equation*}
	(Y^*,Z^*,P^*,Q^*,X^*,B^*)\in\left(\mathbb{H}_T^2\left(\dbP^{\b^*}\right)\right)^6.
	\end{equation*}
	Besides, for all $ t\in[0,T] $, for all $ b\in A $, the optimal control $ \beta^* $ must verify the local incentive constraint:	
	\begin{equation}\label{betaconstraintMFG}
	(Z^*_t + V(t)h(t)\lambda_tP^*_t - \beta^*_t)(b - \beta^*_t)\leq0. 
	\end{equation}
\end{prop}

\ms

\begin{prop}
	Let $ (Y^*,Z^*,P^*,Q^*,X^*,B^*) $ be a solution to the system \eqref{MFGY}-\eqref{MFGO} with $\b^*$ satisfying \eqref{betaconstraintMFG}.  Define $ \xi:= Y_T^*$. Then $ \beta^* $ is an optimal control if for all $ t\in[0,T] $,
	\begin{equation}\label{Qcondition_MF}
	0\leq\frac{Q^*_t}{2h(t)\lambda_t}\leq\inf_{b\in A }\left\{\partial^2_{bb} c(t,b)\right\}.
	\end{equation}
\end{prop}

\subsubsection{Solving the Principal's Problem}

Define the following subsets of $ \mathbb{H}_T^2(\dbP) $:
\beaa
&\ol\cZ:=\Big\{Z :\exists \dbP^\b\in\mathcal{P}~s.t. ~~\mbox{\eqref{MFGY}-\eqref{MFGO} with $\l=\dbE[B]$ have an $\dbF^B$-adapted solution $ (P,Q,X) $ in $ \mathbb{H}_T^2(\dbP^\b)^3 $}\Big\},&\\
&\ul\cZ:=\Big\{Z\in \ol\cZ : ~~\mbox{$Q$ satisfies \eqref{Qcondition_MF}}\Big\},&
\eeaa
and define
\begin{equation*}
\ol{\Upxi} := \left\{  Y_T | Y_0\geq R, Z\in\ol\cZ \right\}\q\mbox{and}\q
\ul{\Upxi} := \left\{  Y_T | Y_0\geq R, Z\in\ul\cZ \right\},
\end{equation*}
where $ Y $ is defined by the equation \eqref{MFGY}.
As before, we have $\ul\Upxi\subset \Upxi\subset\ol{\Upxi} $ and the enlarged optimization of the Principal is given by
\begin{equation}\label{enlargedProblemMFG}
\bar{V}_P :=\sup_{Y_0\geq R}\sup_{Z\in\ol\cZ} \bar J_P(Y_0,Z) := \sup_{Y_0\geq R}\sup_{Z\in\ol\cZ}\mathbb{E}^{\beta^{*}}\left[B_T - Y_T\right].
\end{equation}

\no	Denote
\begin{equation*}
k(t) := h(t)\dbE^{\b^*}[X_t] = h(t) m_0 e^{\int_{0}^{t}\eta(s)ds}\q\mbox{and}\q
\rho_t = e^{\int_{t}^{T}k(s)\mathrm{d}s} - 1.
\end{equation*}

\begin{thm}
	Assume that Assumption \ref{A3} holds true and assume that $ \partial_bc(t,\cdot) $ is surjective. Denote $ \bar{\beta}^*(t) = (\partial_bc(t,\cdot))^{-1}(\rho_t + 1) $. There exists an optimal contract for the Principal's problem in the mean field setting, which is given by
	\begin{equation*}
	\xi^* = \int_{0}^{T}c(t, \bar\b^*(t))\mathrm{d}t + (\bar\b^*(t) - V(t)h(t)\bar B_t P_t)(\mathrm{d}B_t - h(t)\bar B_t \bar{X}_t\mathrm{d}t),
	\end{equation*}
	where $ (\bar{X},\bar{B},V) $ are solutions of the following system of forward equations:
	\begin{numcases}{}
	\mathrm{d}\bar{X}_{t} = \eta(t)\bar X_{t}\mathrm{d}t + h(t)V(t)\bar{B}_t(\mathrm{d}B_t - (h(t)\bar{B}_t\bar{X}_{t} + \bar\b^*(t))\mathrm{d}t) & $ \bar{X}_0 = m_0 $, \\
	\frac{\mathrm{d}\bar{B}_t}{\mathrm{d}t} = k(t)\bar{B}_t + \bar\b^*_t & $ \bar{B}_0 = 0 $,\\
	\dot V(t) = 2\eta(t)V(t) - h^2(t)V^2(t)\bar{B}_t^2 + \sigma(t) & $ V_0 = V_0 $.
	\end{numcases}
\end{thm}
\begin{proof}
	For notation simplicity, denote 
	\begin{equation*}
	\bar{B_t}:=\mathbb{E}^{\beta^*}[B_t],\q
	\bar{\beta}_t = \mathbb{E}^{\beta^*}[\beta^*_t].
	\end{equation*}
	By \eqref{MFGO}, at equilibrium we have 
	\begin{equation}\label{Obardynamic}
	\frac{\mathrm{d}\bar{B_t}}{\mathrm{d}t} = k(t)\bar{B_t}+\bar{\beta}_t.
	\end{equation}
	
	\no Replacing $ B_T $ and $ Y_T = \xi$ in \eqref{enlargedProblemMFG} by the representations in \eqref{MFGY} and \eqref{MFGO}, we get
	\beaa
	\bar{V}_P
	&=& -R + \sup_{Z\in\ol\cZ}\mathbb{E}^{\beta^*}\left[ \int_{0}^{T}\left(h(t)\bar B_tX_t + \beta^*_t-c(t, \beta^*_t)\right)\mathrm{d}t \right] \\
	&\leq& -R + \sup_{Z\in\ol\cZ}\left\{ \int_{0}^{T}\left(k(t) \bar B_t+ \bar\b_t -c(t,\bar\b_t)\right)\mathrm{d}t \right\}.
	\eeaa 
	Notice that the inequality above becomes equality if and only if  $ \beta^* $ is deterministic.
	
	\ms

	\no{ \it Step 1}.\q Solve the deterministic control problem 
	\begin{equation}\label{deterministicControl}
	\sup_{\bar{\beta}}\left\{ \int_{0}^{T}\left(k(t)\bar{B_t} + \bar{\beta}_t-c(t,\bar\b_t)\right)\mathrm{d}t \right\},
	\end{equation}
	where $\bar{B}$ satisfies \eqref{Obardynamic}. It follows from the Pontryagin maximum principle that the optimal control $\bar\b^*$ is:
	\begin{equation}\label{optimalControl_MF}
	\bar{\beta}^*(t) = (\partial_bc(t,\cdot))^{-1}(\rho_t + 1),
	\end{equation}
	where $\rho$ is the adjoint process:
	\begin{equation*}
	\frac{\mathrm{d}\rho_t}{\mathrm{d}t} = -k(t)(\rho_t+1),\q \rho_T=0,
	\end{equation*}
	
	\ms
	
	\no{ \it Step 2}.
	$\bar V_P$ is attained if the optimal control of the Agent is deterministic and is given by \eqref{optimalControl_MF}, now we need to check if such a contract exists and belongs to $\ul\Upxi$. 
	
	Under the control $\bar \b^*(t)$, the process $\bar{B}$ follows:
	\begin{equation*}
	\frac{\mathrm{d}\bar B_t}{\mathrm{d}t} = k(t)\bar B_t + (\partial_bc(t,\cdot))^{-1}(\rho_t + 1).
	\end{equation*}	
	In order for $\bar \b^*(t)$ to be the optimal control of the agent, by \eqref{betaconstraintMFG} we must have
	\begin{equation*}
	Z^*_t  = (\partial_bc(t,\cdot))^{-1}(\rho_t + 1) - V(t)h(t)\bar B_t P_t.
	\end{equation*}
	Inserting the above equality into the equation of $ P $ and we get
	\begin{equation*}
	\mathrm{d}P_{t} = \Big(h(t)\bar B_t\bar\b^*(t) - \eta(t)P_t \Big)\mathrm{d}t + Q_{t}\mathrm{d}I^{\beta^*}_{t},
	\end{equation*}
	and a trivial solution of the BSDE is given by $ (\bar{P}, 0) $, where $ \bar{P} $ is the solution of the  ODE:
	\begin{equation*}
	\frac{\mathrm{d}\bar{P}_{t}}{\mathrm{d}t} = h(t)\bar B_t \bar\b^*(t) - \eta(t)\bar{P}_t \q \text{with} \q \bar P_T=0.
	\end{equation*}
	So far we find a solution to the system \eqref{MFGY}-\eqref{MFGO}. Define the contract:
	\begin{equation*}
	\xi := \int_{0}^{T}c(t, \bar\b^*(t))\mathrm{d}t + (\bar\b^*(t) - V(t)h(t)\bar B_t P_t)(\mathrm{d}B_t - h(t)\bar B_t \bar{X}_t\mathrm{d}t),
	\end{equation*}
	where $ \bar{X} $ is the solution of the following SDE:
	\begin{equation*}
	\mathrm{d}X_{t} = \eta(t)X_{t}\mathrm{d}t + h(t)V(t)\bar B_t (\mathrm{d}B_t - (h(t)\bar B_tX_{t} + \bar \beta^*(t))\mathrm{d}t),\q X_0 = m_0.
	\end{equation*}
	Clearly, $\bar \b^*(t)$ is the best effort of the agent under the contract $\xi$, and	the sufficient condition for the Agent's problem is verified since $ Q=0 $, i.e. $\xi\in \ul\Upxi$. We conclude by using Theorem \ref{principalProblem}.
	\qed
\end{proof}

\section{Appendix}
In this section we shall detail the additional hypothesis and the proof of necessary condition for the Agent's problem in the case of exponential utility in the partially observed linear system. The controlled system is given by \eqref{filterProcess_partial}-\eqref{observeproc_partial}. The Agent's problem that we consider becomes
\begin{equation}\label{eq:agentprob_partial}
V_{A} :=\sup_{\nu\in \cU} J_A(\nu):= \sup_{\nu\in \cU}\mathbb{E}^{\nu}\left[-\exp\left(-\lambda\left(\xi - \int_{0}^{T}c_{t}(\nu_t)\mathrm{d}t\right)\right)\right],
\end{equation}

To guarantee integrability, we shall need additional assumptions on the set of admissible contracts. For convenience, we say that a strictly positive local martingale $ (M_t)_{t\in[0,T]} $ satisfies the condition $ (H) $ if it is uniformly integrable and there exists $ p\in(1,\infty) $ and $ K_p>0 $ constant depending only on $ p $ such that for every stopping time $ \tau $,
\begin{equation}\label{H}
K_p\mathbb{E}[M^{1/p}_T|\mathcal{F}_\tau] \geq M^{1/p}_\tau.\tag{H}
\end{equation}

We recall that a contract $ \xi $ is implementable if there exists at least one optimal control for the Agent's problem. In the rest of the paper, we constrain our study on the following set of contracts:
\begin{equation}\label{admissibleContracts}
\Xi := \left\{ \xi: V_A(\xi)\geq R\text{ and }\left( \mathbb{E}^{\nu*}[\exp(-\lambda\xi)|\mathcal{F}^B_t] \right)_{t\in[0,T]}\text{ satisfied the condition \eqref{H}} \right\},
\end{equation}
where $ \nu^*\in\cM^*(\xi) $.

\begin{rem}\label{rem:condH}
	\begin{enumerate}
		\item Let $ M $ be a martingale and denote $ \mathcal{E}(M) $ the Doléan-Dade exponential. As given in the Corollary 3.4 in \cite{Kazamaki}, $ \mathcal{E}(M) $ satisfying the condition \eqref{H} is equivalent to $ M\in BMO $.
		
		\item Let $ \xi\in\Xi $ and $ \nu^* $ be an optimal control of the Agent's problem. For any control $ \nu\in\cU $, denote
		\begin{equation*}
		M_t := \mathbb{E}^{\nu^*}\left[\exp\left(-\lambda\left(\xi - \int_{0}^{T}c_{s}(\nu_s)\mathrm{d}s\right)\right) \Big| \mathcal{F}^B_t \right].
		\end{equation*}
		Since $ \nu $ is bounded, so is $ \exp(\lambda\int_{0}^{T}c_{s}(\nu_s)\mathrm{d}s) $. One can show easily that $ (M_t)_{t\in[0,T]} $ satisfies the condition \eqref{H}.
	\end{enumerate}
\end{rem}

Define the  Hamiltonian function:
\begin{equation*}
H_t(x, a, b, p, q, z) := \Big(\eta(t)x + a - h(t)V(t)(h(t)x + b)\Big)p - \Big( h(t)x + b \Big)z + c_t(a,b) + \lambda xqz.
\end{equation*}
\begin{rem}
	Compared to the definition of the Hamiltonian functional given in \eqref{eq: hamiltonian}, the additional term $ \lambda xqz $ in the above definition is due to the exponential utility in the Agent's problem. If the Agent is risk-neutral, i.e. $ \lambda=0 $, this term disappears.
\end{rem}

Now we can give the necessary condition for the Agent's problem.

\begin{thm}\label{necessaryCondition_partial}
	For a contract $ \xi\in\Upxi $, let $ \nu^*\in\mathcal{M}^*(\xi) $. Then the following FBSDE
	\begin{numcases}{}
	\mathrm{d}Y_{t} = \Big(c_{t}(\nu^*_t) + \frac{1}{2}\lambda Z^2_t \Big)\mathrm{d}t + Z_{t}\mathrm{d}I^{\nu^*}_{t} & $ Y_{T}=\xi $, \label{stateequation01_partial}  \\
	\mathrm{d}P_{t} = -\partial_x H_t(X_t, \nu^*_t,P_t,Q_t,Z_t)  \mathrm{d}t + Q_{t}\mathrm{d}I^{\nu^*}_{t} & $ P_{T}=0 $,  \label{adjointproc_partial} \\
	\mathrm{d}X_t = \partial_p H_t(X_t, \nu^*_t,P_t,Q_t,Z_t) \mathrm{d}t + h(t)V(t)\mathrm{d}B_t & $ X_0=m_0 $, \label{stateequation02_partial}
	\end{numcases}
	with $I^{\nu^*}$ defined as in \eqref{observeproc_partial}, $V$ defined as in \eqref{eq:varianceproc_partial} and 
	\beaa
	\nu^*_t = (\alpha^*_t,\beta^*_t) = \argmin_{(a,b)\in A }H_t(X_t,a,b,P_t,Q_t,Z_t),
	\eeaa
	has an $\dbF$-adapted solution, denoted by
	\begin{equation*}
	(Y^*,Z^*,P^*,Q^*,X^*) \in \mathbb{H}^2(\mathbb{P}^{\nu^*}) \times BMO\times \mathbb{H}^2(\mathbb{P}^{\nu^*})\times\mathbb{H}^2(\mathbb{P}^{\nu^*})\times \mathbb{H}^2(\mathbb{P}^{\nu^*}).
	\end{equation*}
\end{thm}

\begin{lem}\label{BMO_2}
	Let $ \mathbb{P}^{\nu_1} $ and $ \mathbb{P}^{\nu_2} $ be the two weak control of the Agent and $ Z $ a $ \dbF $-adapted process. If $ \int^{\cdot}_{0}Z_t\mathrm{d}I^{\nu_1}_t\in BMO(\mathbb{P}^{\nu_1}) $, then $ \int^{\cdot}_{0}Z_t\mathrm{d}I^{\nu_2}_t\in BMO(\mathbb{P}^{\nu_2}) $.
\end{lem}

\begin{proof}
	Denote $ X_1 $ and $ X_2 $ as the strong solution of the equation \eqref{filterProcess_partial} given $ \nu_1=(\alpha_1,\beta_1) $ and $ \nu_2=(\alpha_2,\beta_2) $, respectively. Denote $ \Delta X := X_2 - X_1 $ and $ \Delta\beta = \beta_2 - \beta_1 $. The change of measure between $ \mathbb{P}^{\nu_1} $ and $ \mathbb{P}^{\nu_2} $ are given by
	\begin{equation*}
	\mathrm{d}\mathbb{P}^{\nu_2} = \mathcal{E}\left( \int_0^\cdot (h(s)\Delta \hat{X}_s + \Delta\beta_s )\mathrm{d}I^{\nu_1}_s\right)\mathrm{d}\mathbb{P}^{\nu_1}.
	\end{equation*}
	In particular, since the controls are bounded, the martingale $ (\int_{0}^{t}(h(s)\Delta \hat{X}_s + \Delta\beta_t)\mathrm{d}I^{\nu_1}_s)_{t\in[0,T]} $ is in BMO. We conclude by using the Theorem 3.8 in \cite{Kazamaki}.
	\qed
\end{proof}

\no In light of Lemma \ref{BMO_2}, without ambiguity, we say a $ \mathbb{F}^B $-adapted process $ Z\in BMO $ if there exists $ \mathbb{P}^\nu\in\mathcal{P} $ such that $ \int^{\cdot}_{0}Z_t\mathrm{d}I^{\nu}_t\in BMO(\mathbb{P}^{\nu}) $
.

\no Here we introduce a notation: for an $\dbF^B$-adapted process $Z$ satisfying sufficient integrability condition, we denote 
\beaa
\cE^\nu(Z) : = \cE \Big(\int_0^\cd Z_t \mathrm{d} I^\nu_t  \Big).
\eeaa

\no {\bf Proof of Theorem \ref{necessaryCondition_partial}}
	As in the proof of Theorem \ref{necessaryCondition}, we start by assuming the existence of an $ \mathbb{F}^B  $-adapted control $ \nu^{*} $ which optimises the Agent's expected value under a given contract $ \xi $. We denote the corresponding state variable $ \hat{X}^* $.
	
	\no {\it Step 1}.\q Introduce the dynamic version of the value function \begin{equation*}
	\tilde{Y}_{t} := \mathbb{E}^{\nu^*}_t\left[-\exp\left(-\lambda\left(\xi - \int_{t}^{T}c_{s}(\nu^*_s)\mathrm{d}s\right)\right)\right].
	\end{equation*} 
	One can check straightforward that $(M_t)_{0\leq t\leq T}:= \left(\tilde{Y}_{t}\exp\left(\lambda\int_{0}^{t}c_{s}(\nu^*_s)\mathrm{d}s\right)\right)_{0\leq t\leq T}$
	is an $ \mathbb{F}^B $-martingale under the probability $ \mathbb{P}^{\nu^{*}} $, and satisfies the condition \eqref{H} according to Remark \ref{rem:condH}.  By the  martingale representation in Lemma \ref{EMRT}, there exists an $ \mathbb{F}^B  $-adapted process $ \tilde{Z} $, such that for all $ t\in[0,T] $
	\begin{equation}\label{Mdecomp_partial}
	M_t = V_A(\nu^*) + \int_{0}^{t}\tilde{Z}_{s}\mathrm{d}I^{v^{*}}_{s}
	= V_A(\nu^*) + \int_{0}^{t}M_s Z_{s}\mathrm{d}I^{v^{*}}_{s},
	\q\mbox{where}\q  Z_t := -M_t^{-1}\tilde{Z}_t
	\end{equation}
	
	\no Let $ Y_{t} = -\frac{1}{\lambda}\ln(-\tilde{Y}_t) $. Then \eqref{stateequation01_partial} follows from It\^o's formula. Further, in view of \eqref{Mdecomp_partial} we obtain
	\begin{equation*}
	M_t
	=V_A(\nu^*)   \mathcal{E}^{\nu^*}(-\lambda Z).
	\end{equation*}
	Since $M$ satisfies the condition \eqref{H}, we have $ Z \in BMO$ according to Remark \ref{rem:condH} and $ \mathcal{E}^{\nu^*}(-\lambda Z) $ is a uniformly integrable $\dbP^{\nu^*}$-martingale.
	
	\ms
	\no {\it Step 2}. Next we perform a variational calculus around the optimal control $\nu^*$. We shall use the same notation as in the proof of Theorem \ref{necessaryCondition}. Note that $ \Delta \hat{X} $ satisfies the following stochastic differential equation:
	\begin{equation*}
	\mathrm{d}\Delta\hat{X}_t = \left[ (\eta(t)-V(t)h^2(t))\Delta\hat{X}_t + \Delta\alpha_t - V(t)h(t)\Delta\beta_t \right]\mathrm{d}t.
	\end{equation*}
	Since $J_A(\nu^*)=V_A = -e^{-\l Y_0}$,	we have
	\bea
	J_A(\nu^\epsilon) &=&  \mathbb{E}^{\nu^\epsilon}\left[J_A(\nu^*)\mathcal{E}^{\nu^\epsilon}(-\lambda Z)\exp\left(-\lambda\left(\int_{0}^{T}-\delta c_{t}(\nu^*_t) + \epsilon Z_t(h(t)\Delta\hat{X}_t + \Delta\beta_t)\mathrm{d}t\right)\right)\right]. \label{eq:JAexp}
	\eea	
	
	Define $ \mathrm{d}\bar{\mathbb{P}}^{\nu^\epsilon} := \mathcal{E}^{\nu^\epsilon}(-\lambda Z)\mathrm{d}\mathbb{P}^{\nu^\epsilon} $ and denote $ \bar{\mathbb{E}}^{\nu^\epsilon} $ the expectation under $ \bar{\mathbb{P}}^{\nu^\epsilon} $. Applying the Taylor expansion on the exponential function in \eqref{eq:JAexp}, we have
	\beaa
	\frac{J_A(\nu^\epsilon)}{J_A(\nu^*)} 
	&= &\bar{\mathbb{E}}^{\nu^\epsilon}\left[1 - \epsilon\lambda\left(\int_{0}^{T}-\partial_\alpha c_{t}(\nu^*_t)\Delta\alpha_t -\partial_\beta c_{t}(\nu^*_t)\Delta\beta_t + Z_t(h(t)\Delta\hat{X}_t + \Delta\beta_t)\mathrm{d}t\right) + \cR^\eps_T \right],
	\eeaa
	where
	\begin{equation}\label{eq:Reps_partial}
	\cR^\eps_T := (\delta c_{t}(\nu^*_t)-\partial_\alpha c_{t}(\nu^*_t)\delta\alpha_t -\partial_\beta c_{t}(\nu^*_t)\delta\beta_t) + \sum_{n=2}^{\infty}\frac{(-1)^n\epsilon^n\lambda^n}{n!}\left(\int_{0}^{T}-\delta c_{t}(\nu^*_t) + Z_t(h(t)\Delta\hat{X}_t + \Delta\beta_t)\mathrm{d}t\right)^n.
	\end{equation} 
	Define $ \mathrm{d}\bar{\mathbb{P}}^{\nu^*} := \mathcal{E}^{\nu^*}(-\lambda Z)\mathrm{d}\mathbb{P}^{\nu^*} $, $ \bar{\mathbb{E}}^{\nu^*} $ the expectation under $ \bar{\mathbb{P}}^{\nu^*} $, and denote $ \bar{I}^*_t := I^*_t + \int_{0}^{t}\lambda Z_s\mathrm{d}s $. We claim and will prove in {\it Step 3} that $\limsup_{\eps\rightarrow 0} \bar\dbE^{\nu^*}[|\cR^\eps_T|]/\eps^2<\infty$. Then, 
	\begin{equation*}
	\frac{J_A(\nu)}{J_A(\nu^*)} - 1 = \epsilon\lambda\bar{\mathbb{E}}^{\nu^\epsilon}\left[\int_{0}^{T}\big(\partial_\alpha c_{t}(\nu^{*}_{t})\Delta\alpha_t - (Z_t -\partial_\beta c_{t}(\nu^*_t))\Delta\beta_t - Z_t h(t)\Delta\hat{X}_t\big)\mathrm{d}t\right] + \mathcal{O}(\epsilon^2).
	\end{equation*}
	The construction of the costate process $ (P,Q) $ and the variational calculus is the same as step 3 and step 4 of the proof of Theorem \ref{necessaryCondition}. We shall leave them to the readers.
	
	\ms
	\no {\it Step 3}. \q Finally we shall complete the proof by proving the claim we made on $\cR^\eps_T$. 
	The first term in \eqref{eq:Reps_partial} is clearly uniformly bounded by $ C\epsilon^2 $ for some positive constant $ C $, since the controls are bounded.
	%
	Since the controls are bounded, using the H{\"o}lder inequality and the energy inequalities given in \cite[Chapter VII-6]{Meyer}, we have
	\beaa
	\bar\dbE^{\nu^*}[|\cR^\eps_T|] 
	&\leq& C\eps^2+ \sum_{n=2}^{\infty}\frac{\epsilon^n\lambda^n}{n!}\left(C_1^n + C_2^n\bar{\mathbb{E}}^{\nu^*}\left[ \mathcal{E}\left( \int_{0}^{T}\epsilon(h(t)\Delta \hat{X}_t + \Delta\beta_t)\mathrm{d}\bar{I}^*_t \right)\left(\int_{0}^{T}|Z_t|\mathrm{d}t\right)^{n} \right]\right) \\
	&\leq& C\eps^2+\sum_{n=2}^{\infty}\frac{\epsilon^n\lambda^n}{n!}\left(C_1^n + C_2^nC_3\bar{\mathbb{E}}^{\nu^*}\left[\left(\int_{0}^{T}Z^2_t\mathrm{d}t\right)^{n} \right]^{1/2}\right)\\
	&\leq& C\eps^2+\sum_{n=2}^{\infty}\frac{\epsilon^n\lambda^n}{n!}\left(C_1^n + C_2^nC_3\sqrt{n!}||Z||^{n}_{BMO(\bar\dbP^{\nu^*})}\right)
	= \mathcal{O}(\epsilon^2).
	\eeaa
	\qed

\bibliographystyle{plain}

\end{document}